\documentclass[twoside,12pt]{amsart}
\usepackage{amssymb,geometry}
\newtheorem{Proposition}{Proposition}[section]
\numberwithin{equation}{section}
\newtheorem{Lemma}[Proposition]{Lemma}
\newtheorem{Theorem}[Proposition]{Theorem}
\newtheorem{Corollary}[Proposition]{Corollary}
\theoremstyle{definition}
\newtheorem{Definition}[Proposition]{Definition}
\newtheorem{Conjecture}[Proposition]{Conjecture}

\def\part{{\operatorname{part}}}

\def\C{{\mathbb C}}
\def\Z{{\mathbb Z}}

\def\g{\mathfrak g}

\def\s{\mathfrak s}

\def\gl{\mathfrak{gl}}

\def\so{\mathfrak{so}}
\def\sp{\mathfrak{sp}}

\def\part{{\operatorname{part}}}

\title{
A Drinfeld presentation for the twisted Yangian $Y_3^+$
}
\author
{Jonathan Brown}

\address{Department of Mathematics, Computer Science, and Statistics, SUNY Oneonta, Oneonta, NY 13820}
\email{jonathan.brown@oneonta.edu}

\begin{document}

\begin{abstract}
We define the Drinfeld generators for $Y_3^+$, the twisted Yangian associated to the Lie algebra $\so_3(\C)$. 
This allows us to define shifted twisted Yangians, which are certain subalgebras of $Y_3^+$.
We show that there are families of homomorphisms from the shifted twisted Yangians in $Y_3^+$ to 
the universal enveloping algebras of  various orthogonal and symplectic Lie algebras,  and we conjecture that the images of these
homomorphisms are isomorphic to various finite $W$-algebras.
\end{abstract}

\maketitle

\section{Introduction}
In 
\cite{BKshifted} Brundan and Kleshchev define a subalgebra of the Yangian called a shifted Yangian.  
They define shifted Yangians in terms of certain {\em Drinfeld generators} for the Yangian.
The Drinfeld generators were created by Drinfeld in \cite{Dr} in order to study the representation theory of Yangians.
Brundan and Kleshchev then show that
any finite $W$-algebra corresponding to a nilpotent element in $\mathfrak{gl}_n(\C)$ is a homomorphic quotient of a shifted Yangian.
Since then the theory of shifted Yangians has been further developed in \cite{BKrep}, \cite{FMO}, \cite{Pe}, and \cite{KTWWY}.

The main goal of this paper is to start to generalize the results of Brundan and Kleshchev to the other classical Lie algebras.  
Significant progress along these lines is made by the author in \cite{Br} for 
finite $W$-algebras corresponding to a rectangular nilpotent element in classical Lie algebras.
A nilpotent element in a classical Lie algebra is rectangular if its Jordan type is of the form $(l^n)$, i.e. all of its Jordan blocks have the same size.
In this case, it is shown in \cite{Br} that  a finite $W$-algebra corresponding to such a nilpotent is a homomorphic quotient of a twisted Yangian.
In order to generalize this result to arbitrary nilpotent elements in classical Lie algebras we need to first define {\em shifted twisted Yangians};
algebras analogous to the shifted Yangians from \cite{BKshifted}, and to do this we first need to define the Drinfeld generators for twisted Yangians.

This work was motivated by extensive computer calculations to find generators for
various finite $W$-algebras associated to certain nilpotent elements in symplectic and orthogonal Lie algebras. 
As a first step in this paper we find Drinfeld generators for $Y_3^+$, the twisted Yangian associated to the the Lie algebra $\so_3(\C)$,
and we give a presentation of $Y_3^+$ in terms of these generators.

In the following theorems, and throughout this paper except where otherwise noted, 
any occurrence of the letters $a,b,c$ as indices should be interpreted as summing over $\{\pm 1\}$.
So for example, $e_{1,a} e_{a,-1} = e_{1,1}e_{1,-1} + e_{1,-1} e_{-1,-1}$.

\begin{Theorem} \label{T:main1}
    The Yangian $Y_3^+$ 
    has a presentation with
    generators
\[
    \left\{E_{i}^{(m)}, F_j^{(n)}, D_{k,l}^{(r)}, \tilde D_{f,h}^{(s)}, G^{(t)} \mid i,j,k,l,f,h \in \{\pm 1\}, m,n,r,s,t \in \Z_{\geq 0} \right\}
\]
and relations
    \begin{equation} \label{EQ:oD1}
        D_{i,j}^{(0)} = \tilde D_{i,j}^{(0)} = \delta_{i,j}, \quad G^{(0)} = 1, \quad E_i^{(0)} = F_i^{(0)} = 0,
    \end{equation}

    \begin{equation} \label{EQ:oD2}
      \sum_{r=0}^n D_{i,a}^{(r)} \tilde D_{a,j}^{(n-r)} = \delta_{n,0} \delta_{i,j},
    \end{equation}
    
    \begin{equation} \label{EQ:GG}
       [G^{(m)}, G^{(n)}] = 0,
    \end{equation}

    \begin{align} \label{EQ:oDD}
        [D_{i,j}^{(m)}, D_{k,l}^{(n)}] &= 
            \sum_{r=0}^{m-1} D_{k,j}^{(m-1-r)} D_{i,l}^{(n+r)} - D_{k,j}^{(n+r)} D_{i,l}^{(m-1-r)} \\ \notag
            &\quad - \sum_{r=0}^{m-1} (-1)^r \left(  D_{i,-k}^{(m-1-r)} D_{-j,l}^{(n+r)} - D_{k,-i}^{(n+r)} D_{-l,j}^{(m-1-r)} \right) \\   \notag
            &\quad + \sum_{r=0}^{\lfloor{m/2}\rfloor-1} D_{k,-i}^{(m-2-2r)} D_{-j,l}^{(n+2r)} - 
                       D_{k,-i}^{(n+2r)} D_{-j,l}^{(m-2-2r)},
    \end{align}

    \begin{equation} \label{EQ:oDs}
         D_{i,j}^{(n)} = (-1)^n D_{-j,-i}^{(n)} + \tfrac{1}{2} \left( D_{-j,-i}^{(n-1)} - (-1)^n D_{-j,-i}^{(n-1)} \right),
    \end{equation}
    
    \begin{equation} \label{EQ:FEs}
         F_i^{(m)} = (-1)^m \sum_{r=1}^m 2^{m-r} \binom{m-1}{m-r} E_{-i}^{(r)},
    \end{equation}
    
    \begin{equation} \label{EQ:EFs}
         E_i^{(m)} = (-1)^m \sum_{r=1}^m 2^{m-r} \binom{m-1}{m-r} F_{-i}^{(r)},
    \end{equation}

    \begin{equation} \label{EQ:DEreln}
   [D_{i,j}^{(m)}, E_{k}^{(n)}] = \delta_{j,k} \sum_{r=0}^{m-1} D_{i,a}^{(r)} E_{a}^{(m+n-1-r)} 
   - \delta_{i,-k} \sum_{r=0}^{m-1} \sum_{s=0}^r (-2)^s (-1)^{r-s} 
   \begin{pmatrix} r\\ s \end{pmatrix}
   E_{-a}^{(n+r-s)} D_{a,j}^{(m-1-r)},
\end{equation}

    \begin{equation} \label{EQ:DFreln}
   [D_{i,j}^{(m)},F_k^{(n)}] = -\delta_{i,k} \sum_{r=0}^{m-1}  F_{a}^{(m+n-1-r)} D_{a,j}^{(r)} 
   + \delta_{j,-k} \sum_{r=0}^{m-1} \sum_{s=0}^r (-2)^s (-1)^{r-s} 
   \begin{pmatrix} r\\ s \end{pmatrix}
    D_{i,a}^{(m-1-r)} F_{-a}^{(n+r-s)},
\end{equation}

\begin{equation} \label{EQ:GEreln}
   [G^{(m)},E_{i}^{(n)}] = - \sum_{r=0}^{m-1}  G^{(r)} E_{i}^{(m+n-1-r)} 
   + \sum_{r=0}^{m-1} \sum_{s=0}^r (-2)^s (-1)^{r-s} 
   \begin{pmatrix} r \\ s \end{pmatrix}
    E_{i}^{(n+r-s)} G^{(m-1-r)},
\end{equation}

\begin{equation} \label{EQ:GFreln}
   [G^{(m)},F_{i}^{(n)}] = \sum_{r=0}^{m-1}  F_{i}^{(m+n-1-r)} G^{(r)}
   - \sum_{r=0}^{m-1} \sum_{s=0}^r (-2)^s (-1)^{r-s} 
   \begin{pmatrix} r \\ s \end{pmatrix}
   G^{(m-1-r)} F_{i}^{(n+r-s)},
\end{equation}

\begin{align} \label{EQ:EFreln}
   [E_{i}^{(m)}, F_{j}^{(n)}] &= - \sum_{r=0}^{n+m-1} G^{(r)} \tilde D_{i,j}^{(n+m-1-r)} \\ \notag
     &\quad 
     - \sum_{r=0}^{m-1} 
     \sum_{s=0}^r 
     (-2)^s(-1)^{r-s}
     \begin{pmatrix}
          r \\ s
     \end{pmatrix}
       \left( E_{i}^{(m-r-1)} F_{j}^{(n+r-s)} + F_{-i}^{(n+r-s)} E_{-j}^{(m-r-1)} \right) \\ \notag
       &\quad + 
       \sum_{r=0}^{m-1} 
     (-2)^r(-1)^{m-1-r}
     \begin{pmatrix}
          m-1 \\ r
     \end{pmatrix}
       \sum_{s=0}^{m+n-1-r}
      F_{-i}^{(s)} F_{j}^{(m+n-1-r-s)},
\end{align}

\begin{align} \label{EQ:EEreln}
   [E_{i}^{(m)}, E_{j}^{(n)}] &= 
   \sum_{r=0}^{n-1} E_{j}^{(m+n-1-r)} E_{i}^{(r)}
   - \sum_{r=0}^{m-1} E_{j}^{(m+n-1-r)} E_{i}^{(r)}
   \\ \notag
       &\quad - 
       \sum_{r=0}^{m-1} 
     (-2)^r(-1)^{m-1-r}
     \begin{pmatrix}
          m-1 \\ r
     \end{pmatrix}
      \sum_{s=0}^{m+n-1-r}
      \tilde D_{j,-i}^{(s)} G^{(m+n-1-r-s)}.
\end{align}

\begin{align} \label{EQ:FFreln}
   [F_{i}^{(m)}, F_{j}^{(n)}] &= 
   \sum_{r=0}^{m-1} F_{j}^{(r)} F_{i}^{(m+n-1-r)} 
   - \sum_{r=0}^{n-1} F_{j}^{(r)} F_{i}^{(m+n-1-r)} 
   \\ \notag
       &\quad - 
       \sum_{r=0}^{n-1} 
     (-2)^r(-1)^{n-1-r}
     \begin{pmatrix}
          n-1 \\ r
     \end{pmatrix}
      \sum_{s=0}^{m+n-1-r}
      \tilde D_{-j,i}^{(s)} G^{(m+n-1-r-s)}.
\end{align}

\end{Theorem}

A subset of these generators forms a PBW basis for $Y_3^+$.
We say a triple $(i,j, r)$ of integers is {\em admissible}  if $i +j < 0$ if $r$ is odd and $i+j \leq 0$ if $r$ is even.

\begin{Theorem} \label{T:PBWI}
The set of monomials in 
\[
    \left\{E_{i}^{(m)}, D_{k,l}^{(r)}, G^{(t)} \mid i,k,l \in \{\pm 1\}, m > 0, r > 0, (k,l,r) \text{ is admissible}, t \geq 0, t \text{ is even} \right\}
\]
taken in some fixed order forms a PBW basis of $Y_3^+$.
\end{Theorem}
The theorem also holds if the 
$E_i^{(m)}$'s 
are replaced with
$F_i^{(m)}$'s.

The key motivation in finding this presentation is to define the {\em shifted twisted Yangians}, which we hope to relate to 
certain finite $W$-algebras.
   
Let $k > 0$.
\begin{Definition}
The  {\em $k$-shifted twisted Yangian for $\so_3(\C)$}, denoted $Y_3^+(k)$,
is the subalgebra of $Y^+_3$ generated by $\{D_{i,j}^{(m)}, G^{(n)}, E^{(r)}_l \mid
    i,j,l \in \{\pm 1\}, m,n \in \Z_{\geq 0}, r > k\}$.
\end{Definition}

It is easy to see from Theorems \ref{T:main1} and \ref{T:PBWI}  that these are in fact proper subalgebras of $Y_3^+$.
In particular note that all the elements 
$E_i^{(r)}$ and$ E_j^{(r)}$
where $r \leq k$
in \eqref{EQ:EEreln} cancel, provided that $m,n > k$.

One remarkable thing about these shifted twisted Yangians are
the partial-evaluation homomorphisms defined in $\S$\ref{S:partialeval}.
We summarize these in the following theorem, see $\S$\ref{S:partialeval}
for the full definitions.

\begin{Theorem}
	For each $k  > 0$
	there exists an injective homomorphism
	$\phi_k : Y_3^+(k) \to Y_3^+(k-1)\otimes U(\gl_1(\C))$.
\end{Theorem}

Fix a positive integers $k,n$,
and
let $\g = \sp_{3n+2k} (\C)$ if $n$ is even, and let $\g = \so_{3n+2k}(\C)$ if $n$ is odd.  
 Let $e \in \g$ be a nilpotent element with 
 Jordan Type $(n+2k,n,n)$, and let $U(\g,e)$ 
 denote the finite $W$-algebra associated to $e$ and $\g$ (see \cite{P1} or \cite{BGK}for the definition of finite $W$-algebras).

 The partial-evaluations homomorphisms, along with the homomorphisms $\kappa_l$
 from \cite[(1.12)]{Br}
 can be used to construct an algebra homomorphism $\phi : Y^+_3(k) \to U(\g)$ (see $\S$\ref{S:partialeval} for details).

 \begin{Conjecture}
 The algebras $\phi(Y_3^+(k))$ and $U(\g,e)$ are isomorphic.
 \end{Conjecture}
 Once proven, this conjecture will lead to a presentation of $U(\g,e)$.
 Finite $W$-algebras associated to nilpotent elements with Jordan Type $(m,n,n)$ are called three-row finite $W$-algebras.  We expect these
 to play an important role in the representation theory of large classes of finite $W$-algebras associated to more complicated nilpotent orbits.
 For example, in \cite{B2} the finite dimensional irreducible representations of rectangular finite $W$-algebras are classified.
 This classification led to the results of \cite{BG1}, which are used with the results from \cite{BG2} to classify  
 all of the finite dimensional irreducible representations
 of all type B and C finite $W$-algebras associated to standard Levi nilpotent orbits
 in \cite{BG3} and \cite{BG4}.
 We expect the classification of the irreducible representations of three row finite $W$-algebras
 will play a similar role to that of the rectangular finite $W$-algebras in 
 the study of the representation theory of more complicated finite $W$-algebras.

 Finally in $\S$\ref{S:center} we express generators of the center of $Y_3^+$ in terms of the Drinfeld generators.

\section{ Yangians and twisted Yangians}
Throughout this section
we work over the index set $\mathcal{I}_n$.
If $n$ is even, say $n=2k$, then
\[
   \mathcal{I}_n = \{-k, -k+1, \dots, -1, 1, \dots, k \}.
\]

If $n$ is odd , say $n=2k+1$, then
\[
   \mathcal{I}_n = \{-k, -k+1, \dots, -1, 0, 1, \dots, k \}.
\]

The Yangian $Y_n$ is an algebra with countably many generators
$\{T_{i,j}^{(r)}\ | i,j \in \mathcal{I}_n, r \in \Z_{>0}\}$.
Let $u,v$ be an indeterminants.
By letting
\[
  T_{i,j}(u) = \sum_{r \ge 0} T_{i,j}^{(r)}u^{-r} \in Y_n[[u^{-1}]],
\]
the relations are given by
\[
[T_{i,j}(u), T_{k,l}(v)] = \frac{1}{u-v} (T_{k,j}(u) T_{i,l}(v) - T_{k,j}(v) T_{i,l}(u)).
\]
This (and subsequent formulas) involving generating functions
should be interpreted by using that
\[
   \frac{1}{u-v} = u^{-1} \sum_{r=0}^\infty u^{-r} v^r
\]
and
equating coefficients of the indeterminants $u,v$ on both sides of equations,
where the both sides of the above equation are elements in  $Y_n[[u^{-1}, v^{-1}]]$ localized at the non-zero elements of $\C[[u^{-1}, v^{-1}]]$.

There exists a homomorphism $\operatorname{ev} : Y_n \to U(\gl_n(\C))$ defined via $\operatorname{ev}(T_{i,j}^{(r)}) = \delta_{r,1} e_{i,j}$.


By  \cite[$\S$3.4]{MNO}, there exists an 
involution $\tau:Y_n \rightarrow Y_n$
defined by 
\[
\eta(T_{i,j}(u)) = T_{-j,-i}(-u).
\]
Now $Y_n^+$, the twisted Yangian associated to $\so_n(\C)$, is defined to be the subalgebra of $Y_n$ 
generated by the elements $\{S_{i,j}^{(r)} \:|\: i,j \in \mathcal{I}_n, r \in \Z_{>0}\}$
coming from the expansion
\begin{equation}\label{siju}
  S_{i,j} (u) = \sum_{r \geq 0} S_{i,j}^{(r)} u^{-r}
= 
\sum_{k \in \mathcal{I}_n} 
    \eta(T_{i,k}(u)) T_{k,j}(u) \in Y_n[[u^{-1}]].
\end{equation}

This in turn allows us to define the $S$-matrix $S(u) \in M_{n}(Y_n^+[[u^{-1}]])$
by defined $S(u)_{i,j} = S_{i,j}(u)$, where $M_{n}(R)$ denotes the ring of $n \times n$ matrices with entries in a ring $R$
and with entries indexed by $\mathcal{I}_n$.

Let 
\[
  R(u) = u - \sum_{i,j \in \mathcal{I}_n} e_{i,j} \otimes e_{j,i} \in M_{n}(\C) \otimes M_{n}(\C)(u),
\]
and let 
\[
 R'(u) = u  - \sum_{i,j \in \mathcal{I}_n} e_{-i,-j} \otimes e_{i,j}.
\]

Let $R^{[1,2]}(u) \in M_{n}(\C) \otimes M_{n}(\C) \otimes Y_n^+[[u^{-1}]][u]$ denote the operator $R(u)$ acting in the first 2 tensor powers,
and define $R'^{[1,2]}$ similarly.
Also let $S^{[1,3]}(u)$ denote the $S$-matrix ``acting'' in the first and third tensor powers, specifically
$S^{[1,3]}(u) = \sum_{i \in \mathcal{I}_n} e_{i,j} \otimes 1 \otimes S_{i,j}(u)$.
We define $S^{[2,3]}(u)$ similarly.
Also let $S^T(u)$ denote the transpose of $S(u)$ defined by
\[
S^T(u) = \sum_{i,j \in \mathcal{I}_n} e_{-j,-i} \otimes S_{i,j}(u).
\]

Now the relations for $Y_n^+$ are induced by the equations
\begin{equation} \label{tyreln1}
   R^{[1,2]}(u-v) S^{[1,3]}(u) R'^{[1,2]}(-u-v) S^{[2,3]}(v) = S^{[2,3]}(v) R'^{[1,2]}(-u-v) S^{[1,3]}(u) R^{[1,2]}(u-v)
\end{equation}
and
\begin{equation} \label{tyreln2}
   S^T(-u) = S(u) + \frac{S(u) -S(-u)}{2 u}.
\end{equation}
with the equality from \eqref{tyreln1} occurring in 
$M_{n}(\C) \otimes M_{n}(\C) \otimes Y_n^+((u^{-1}, v^{-1}))$, 
the localization of 
$M_{n}(\C) \otimes M_{n}(\C) \otimes Y_n^+[[u^{-1}, v^{-1}]]$ at the multiplicative set consisting of the non-zero elements of
$\C[u^{-1},v^{-1}]$.

By computing the coefficient of $e_{i,j} \otimes e_{k,l}$ in \eqref{tyreln1} we get the more explicit relations
\begin{align} \label{EQ:SS}
  (u^2-v^2)[S_{i,j}(u),S_{k,l}(v)] &= (u+v)(S_{k,j}(u) S_{i,l}(v) - S_{k,j}(v) S_{i,l}(u)) \\ \notag
  & \quad - (u-v) (S_{i,-k}(u) S_{-j,l}(v) - S_{k,-i}(v) S_{-l,j}(u)) \\ \notag
  & \quad + S_{k,-i}(u) S_{-j,l}(v) - S_{k,-i}(v) S_{-j,l}(u).
\end{align}
This in turn can be used to give an explicit formula for $[S_{i,j}^{(r)},S_{k,l}^{(s)}]$ by using the expansions
\begin{align*}
\frac{1}{u-v} &= \sum_{r=0}^\infty u^{-1} (u^{-1} v^{-1})^r, \\
\frac{1}{u+v} &= \sum_{r=0}^\infty (-1)^r u^{-1} (u^{-1} v^{-1})^r, \\
\frac{1}{u^2-v^2} &= \sum_{r=0}^\infty u^{-2} (u^{-2} v^{-2})^r, \\
\end{align*}.

Let $\tilde S^{[2,3]}(v)$ denote the inverse of $S^{[2,3]}(v)$.
We obtain an alternate form of the relations by multiplying both sides of \eqref{tyreln1}  on the left and the right by $\tilde S^{[2,3]}(v)$ to obtain
\begin{equation} \label{tyrelnalt}
   \tilde S^{[2,3]}(v) R^{[1,2]}(u-v) S^{[1,3]}(u) R'^{[1,2]}(-u-v) = R'^{[1,2]}(-u-v) S^{[1,3]}(u) R^{[1,2]}(u-v)\tilde S^{[2,3]}(v).
\end{equation}

There is an important antiautomorphism of the twisted Yangian, $\tau$, defined via
\begin{equation} \label{EQ:tau}
    \tau(S_{i,j}^{(r)}) = S_{-j,-i}^{(r)}.
\end{equation}
Note that $\tau$ is simply the map induced by $S(u) \mapsto S^T(u)$.

The Yangian $Y_n^+$ is closely related to $U(\so_{n}(\C))$.
To see this
we realize $\so_{n}(\C)$ as a subalgebra of $\gl_{n}(\C)$ in the following manner.  We consider $\gl_{n}(\C)$ to have basis the matrix
units $\{e_{i,j} \mid i,j \in \mathcal{I}_n\}$. Now we realize $\so_{n}(\C)$ as the Lie subalgebra of 
$\gl_{n}(\C)$ generated by $\{f_{i,j} = e_{i,j} - e_{-j,-i} \mid i,j \in \mathcal{I}_n\}$.
\begin{Theorem}[Molev]
The map induced by $f_{i,j} \mapsto S_{i,j}^{(1)}$ defines an injective homomorphism
\[
  \iota : U(\so_{n}(\C)) \hookrightarrow Y_n^+.
\]
The map induced by $S_{i,j}^{(r)} \mapsto \delta_{r,1} f_{i,j}$ defines a surjective homomorphism
\[
  \rho : Y_n^+ \twoheadrightarrow U(\so_{n}(\C)).
\]
\end{Theorem}

There is also a twisted Yangian associated to $\sp_n(\C)$, however since we do not use it in this paper we are not including its definition.

\section{A Drinfeld presentation of $Y_3^+$} \label{S:IDP}
In this section we define Drinfeld generators for $Y_3^+$ and prove Theorem \ref{T:main1}.
We follow the approach of Brundan and Kleshchev in \cite[$\S$5]{BKparabpres}.

To define the Drinfeld generators, we
define the matrix $S(u)$ to have rows and columns indexed by $\{-1,1,0\}$, and $ij$ entry equal
to $S_{i,j}(u)$.  Note the unusual order of the indices.
Now $S(u)$ has a Guass factorization
\begin{equation} \label{EQ:Smatrix}
S(u) = 
\begin{pmatrix}
 1 & 0 \\ F(u) & 1
 \end{pmatrix}
\begin{pmatrix}
 D(u) & 0 \\ 0 & G(u) 
 \end{pmatrix}
\begin{pmatrix}
 1 & E(u) \\ 0 & 1
 \end{pmatrix}
\end{equation}
for some matrices $D(u), E(u), F(u)$, and $G(u)$ where
$D(u)$ is a $2 \times 2$ matrix,
$E(u)$ is an $2 \times 1$ matrix,
$F(u)$ is a $1 \times 2$ matrix,
and $G(u)$ is a $1 \times 1$ matrix.
We consider all of these matrices to inherit their indices from $S(u)$, so eg
\[
  E(u) = \begin{pmatrix} E_{-1}(u) \\ E_{1}(u) \end{pmatrix}.
\]
We let $\tilde D(u)$ and $\tilde G(u)$ denote the inverses of 
$D(u)$ and $G(u)$ respectively.

Explicit formulas for $D(u), E(u), F(u)$, and $G(u)$ are now given by
\begin{equation} \label{EQ:DS}
D(u) = 
\begin{pmatrix}
S_{-1,-1}(u) & S_{-1,1}(u) \\
S_{1,-1}(u) & S_{1,1}(u)
\end{pmatrix},
\end{equation}
\begin{equation} \label{EQ:ES}
E(u) = \tilde D (u) 
\begin{pmatrix} S_{-1,0}(u)  \\ S_{1,0}(u) \end{pmatrix},
\end{equation}
\begin{equation} \label{EQ:FS}
F(u) =  
\begin{pmatrix} S_{0,-1}(u)  & S_{0,1}(u) \end{pmatrix}
\tilde D(u),
\end{equation}
and
\begin{equation} \label{EQ:GS}
 G(u) = S_{0,0}(u) - 
\begin{pmatrix} S_{0,-1}(u)  & S_{0,1}(u) \end{pmatrix}
 \tilde D(u)
\begin{pmatrix} S_{-1,0}(u)  \\ S_{1,0}(u) \end{pmatrix}.
\end{equation}

\begin{Theorem} \label{T:1}
   The following hold for all $i,j,k,l \in \{\pm 1\}$.
   \begin{enumerate}
   \item
	\label{EQ:D}
\begin{align*} 
	[D_{i,j}(u),D_{k,l}(v)] &= \frac{1}{u-v}(D_{k,j}(u) D_{i,l}(v) - D_{k,j}(v) D_{i,l}(u)) \\
	\notag & \quad - \frac{1}{u+v }(D_{i,-k}(u) D_{-j,l}(v) - D_{k,-i}(v) D_{-l,j}(u))   \\
	\notag  & \quad + \frac{1}{u^2-v^2} D_{k,-i}(u) D_{-j,l}(v) - D_{k,-i}(v) D_{-j,l}(u),
\end{align*}

    \item
	\label{EQ:DE}
\begin{align*} 
   [D_{i,j}(u), E_{k,0}(v)] &= 
      \frac{\delta_{j,k}}{u-v} D_{i,a}(u)(E_{a,0}(v)- E_{a,0}(u)) \\
\notag       & \quad + \frac{\delta_{i,-k}}{u+v+2} ( F_{0,a}(u) - E_{-a,0}(v))D_{a,j}(u),
\end{align*}

\item
	\label{EQ:FDI}
\begin{align*} 
   [D_{i,j}(u),
   F_{0,k}(v)]
   & = 
       -\frac{\delta_{i,k}}{u-v} (F_{a,0}(v)- F_{a,0}(u))D_{a,j}(u) \\ \notag
       & \quad - \frac{\delta_{j,-k}}{u+v+2} D_{i,a}(u) ( E_{a,0}(u) - F_{0,-a}(v)) 
\end{align*}
\item
	\label{EQ:GEI}
\begin{align*} 
     [G(u), E_{i,0}(v)] &=  \frac{1}{u-v}G(u) (E_{i,0}(u) - E_{i,0}(v)) \\ \notag
     &\quad + \frac{1}{u+v+2} (E_{i,0}(v) -F_{0,-i}(u))G(u), 
     \end{align*}

\item
	\label{EQ:GFactual} \label{EQ:GFI}
\begin{align*} 
     [F_{0,i}(u), G(v)] &=  \frac{1}{u-v} (F_{0,i}(u) - F_{0,i}(v))G(v)  \\
     &\quad + \frac{1}{u+v+2} G(v) (F_{0,i}(u) -E_{-i,0}(v)),
     \end{align*}

\item
	\label{EQ:EFI}
\begin{align*} 
      [E_{i,0}(u), F_{0,j}(v)] &= 
          \frac{1}{u-v} ( \tilde D_{i,j}(u)G(u) - G(v) \tilde D_{i,j}(v) ) \\ \notag
          &+\frac{1}{u+v+2} (E_{i,0}(u) - F_{0,-i}(v))(E_{-j,0}(u) - F_{0,j}(v)),
\end{align*}

\item
	\label{EQ:EEI}
\begin{align*} 
 [E_{i,0}(u),E_{j,0}(v)] & =
      \frac{1}{u-v} (E_{i,0}(u) - E_{i,0}(v))(E_{j,0}(u) - E_{j,0}(v)) \\ \notag
      &\qquad
         + \frac{1}{u+v+2} \tilde D_{j,-i}(u) G(u) \\ \notag
      &\qquad - \frac{u-v}{(u-v-1)(u+v+2)} \tilde D_{j,-i}(v) G(v) \\ \notag
      &\qquad + \frac{1}{(u-v-1)(u+v+2)} \tilde D_{i,-j}(v) G(v) \\ \notag
      &\qquad
         - \frac{1}{(u-v-1)(2v+3)} (\tilde D_{i,-j}(v) G(v) - \tilde D_{j,-i}(v) G(v)),
\end{align*}
\item
	\label{EQ:FFI}
\begin{align*} 
 [F_{0,i}(u),F_{0,j}(v)] & =
      -\frac{1}{u-v} (F_{0,j}(u) - F_{0,j}(v))(F_{0,i}(u) - F_{0,i}(v)) \\ \notag
      &\qquad 
         - \frac{1}{u+v+2} \tilde D_{-i,j}(u) G(u) \\ \notag
      &\qquad + \frac{u-v}{(u-v-1)(u+v+2)} \tilde D_{-i,j}(v) G(v) \\ \notag
      &\qquad - \frac{1}{(u-v-1)(u+v+2)} \tilde D_{-j,i}(v) G(v) \\ \notag
      &\qquad
         + \frac{1}{(u-v-1)(2v+3)} (\tilde D_{-j,i}(v) G(v) - \tilde D_{-i,j}(v) G(v)),
\end{align*}
\item
	\label{EQ:DsymI}
\begin{equation*} 
D_{i,j}(u) = D_{-j,-i}(-u) + \frac{1}{2u} \left(D_{-j,-i}(u) - D_{-j,-i}(-u)\right),
\end{equation*}

\item
	\label{EQ:EFsymm}
\begin{equation*} 
  E_i(-u) = F_{-i}(u-2) 
\end{equation*}

\item
	\label{EQ:FEsymm}
	\[
  F_i(-u) = E_{-i}(u-2)
  \]

\end{enumerate}

\end{Theorem}

The proof of this theorem is rather long and technical, so we delay it until $\S$\ref{S:proof}

From Theorem \ref{T:1} we can deduce the relations in Theorem \ref{T:main1}
by computing the coefficients of $u^{-r} v^{-s}$ on both sides of the equations in Theorem \ref{T:main1}.
The following power series expansions will help:
\[
   \frac{1}{u-v} = \sum_{k \geq 0} u^{-(k+1)} v^k,
\]
\[
   \frac{1}{u+v} = \sum_{k \geq 0} (-1)^k u^{-(k+1)} v^k,
\]
and
\begin{equation} \label{EQ:uv2}
   \frac{1}{u+v+2} 
   = \sum_{k \geq 0} (-2)^k(u+v)^{-(k+1)}
   = \sum_{k \geq 0} (-2)^k \left(\sum_{l \geq 0} (-1)^l u^{-(l+1)} v^l \right)^{k+1}.
\end{equation}
It will also be useful to recall 
the classical formula
\begin{equation} \label{EQ:xcoeff}
\left(\sum_{r=0}^\infty x^r \right)^k
=
\sum_{n=0}^\infty \binom{n+k-1}{k-1} x^n.
\end{equation}

We demonstrate how two of the relations in Theorem \ref{T:main1}
follow from Theorem \ref{T:1}, and leave the rest to the reader.
\subsection{Proof of relation \eqref{EQ:FEs}}

\begin{proof}
Substituting $-u$ for $u$ in Theorem \ref{T:1} \eqref{EQ:FEsymm} gives
\[
 F_i(u) = E_{-i} (-(u+2)) = 
 \sum_{r=1}^\infty (-1)^r E_{-i}^r (u+2)^{-r} =
 \sum_{r=1}^\infty (-1)^r E_{-i}^r \left(\sum_{s=0}^\infty (-2)^s u^{-(s+1)} \right)^r.
\]
Now by examining \eqref{EQ:xcoeff}, and looking at the $u^{-m}$ coefficient on both sides of this we get
\[
   F_i^{(m)} = \sum_{r=1}^m (-1)^r (-2)^{m-r} \binom{ m-1}{r-1}  E_{-i}^{(r)}.
\]

\end{proof}

\subsection{Proof of relation \eqref{EQ:DEreln}}
\begin{proof}
	We calculate the coefficients of $u^{-m} v^{-n}$ on both sides
	of Theorem \ref{T:1} \eqref{EQ:DE}.
	First we will calculate the coefficient of $u^{-m} v^{-n}$ in
	\[
      \frac{\delta_{j,k}}{u-v} D_{i,a}(u)(E_{a,0}(v)- E_{a,0}(u)) 
   \]
   Since
   \[ \frac{1}{u-v} = \sum_{k \geq 0}  u^{-(k+1)} v^k
   \]
   had only positive powers of $v$, the 
   \[
	   -\frac{\delta_{j,k}}{u-v}D_{i,a}(u) E_{a,0}(u)  
   \]
   term of this contributes nothing to the $u^{-m} v^{-n}$ coefficient.  
   So we need to only calculate the $u^{-m} v^{-n}$ coefficient of
   \[
      \frac{\delta_{j,k}}{u-v} D_{i,a}(u)E_{a,0}(v).
      \]
      Note that the $u^{-m}$ coefficient of this is
     \[
	     \delta_{j,k} \left(\sum_{r=0}^{m-1} D_{i,a}^{(m-r-1)} v^r\right) E_{a,0}(v),
     \]
     so the $u^{-m} v^{-n}$ coefficient is
     \[
	     \delta_{j,k} \sum_{r=0}^{m-1} D_{i,a}^{(m-r-1)} E_{a,0}^{(n+r)} 
	     = 
	     \delta_{j,k} \sum_{r=0}^{m-1} D_{i,a}^{(r)} E_{a,0}^{(m+n-r-1)}.
     \]

   Now we need to calculate the $u^{-m} v^{-n}$ coefficient of
   \[
 \frac{\delta_{i,-k}}{u+v+2} ( F_{0,a}(u) - E_{-a,0}(v))D_{a,j}(u).
 \]
 Again our expansion of $\tfrac{1}{u+v+2}$ from \eqref{EQ:uv2} has only
 positive powers of $v$, so the 
   \[
 \frac{\delta_{i,-k}}{u+v+2}  F_{0,a}(u)D_{a,j}(u).
 \]
 term contributes nothing to the coefficient of $u^{-m} v^{-n}$,
 so we only need to calculate the $u^{-m} v^{-n}$ coefficient of
   \[
  -\frac{\delta_{i,-k}}{u+v+2}  E_{-a,0}(v)D_{a,j}(u).
 \]
 Using \eqref{EQ:uv2} and \eqref{EQ:xcoeff} we calculate the the $u^{-r}$ coefficient of
 $\tfrac{1}{u+v+2}$ is 
 \[  
	 \sum_{s=0}^{r-1} (-2)^s (-1)^{r-s-1} \begin{pmatrix} r-1 \\s \end{pmatrix} v^{r-s-1},
 \]
 thus the $u^{-m} v^{-n}$ coefficient of 
   \[
  -\frac{\delta_{i,-k}}{u+v+2}  E_{-a,0}(v)D_{a,j}(u).
 \]
 is
 \begin{align*}
	 -\delta_{i,-k} & \sum_{r=1}^m \sum_{s=0}^{r-1} (-2)^s (-1)^{r-s-1} \begin{pmatrix} r-1 \\s \end{pmatrix} E_{-a,0}^{(n+r-s-1)} D_{a,j}^{(m-r)}
		 = \\
 & \qquad \qquad
		 -\delta_{i,-k} \sum_{r=0}^{m-1} \sum_{s=0}^{r} (-2)^s (-1)^{r-s} \begin{pmatrix} r \\s \end{pmatrix} E_{-a,0}^{(n+r-s)} D_{a,j}^{(m-r-1)}.
		 \end{align*}

\end{proof}

\subsection{A PBW basis in terms of Drinfeld generators}
In order to prove Theorem \ref{T:PBWI} we need to recall the PBW basis for $Y_3^+$ from \cite[Remark 3.14]{MNO}:
Monomials in
$
   \{ S_{i,j}^{(r)} \mid (i,j,r) \text{ is admissible} \}
$
form a basis of $Y_3^+$, where 
$(i,j,r)$ is admissible if $i+j < 0$ if $r$ is odd and $i+j \leq 0$ if $r$ is even.

There is a filtration on $Y_3^+$ formed by declaring that $S_{i,j}^{(r)}$ is in degree $r-1$.
Let 
\[
L_r Y_3^+ = \{ x \in Y_3^+ \mid \text{ the degree of $x$ is at most $r$}\}
\]
\begin{Lemma}
For all $i \in \{\pm 1\}$ and $r >0$
\[
      E_{i,0}^{(r)} \equiv S_{i,0}^{(r)} \mod L_{r-1} Y_3^+,
\]
\[
      F_{0,i}^{(r)} \equiv S_{0,i}^{(r)} \mod L_{r-1} Y_3^+,
\]
\[
      G^{(r)} \equiv S_{0,0}^{(r)} \mod L_{r-1} Y_3^+,
\]
and
\[
 E_{i,0}^{(r)} \equiv F_{0,-i}^{(r)} \mod L_{r-1} Y_3^+.
 \]
\end{Lemma}
\begin{proof}
 This follows immediately from \eqref{EQ:ES}, \eqref{EQ:FS}, \eqref{EQ:GS}, and \eqref{tyreln2}.
\end{proof}
Now Theorem \ref{T:PBWI} follows immediately from this lemma and the definition of the usual PBW basis for $Y_3^+$ given above.

\section{Proof of Theorem \ref{T:1}} \label{S:proof}
Recall the alternate form of the twisted Yangian relations from \eqref{tyrelnalt}:
\begin{equation} \label{tyrelnalt2}
   \tilde S^{[2,3]}(v) R^{[1,2]}(u-v) S^{[1,3]}(u) R'^{[1,2]}(-u-v) = R'^{[1,2]}(-u-v) S^{[1,3]}(u) R^{[1,2]}(u-v)\tilde S^{[2,3]}(v),
\end{equation}
where $\tilde S(u)$ denotes the inverse of $S(u)$.

By calculating the coefficient  of $e_{i,j} \otimes e_{k,l}$ on both sides of this equation, we get
that for all $i,j,k,l \in \mathcal{I}$
\begin{align} \label{EQ:relhelp}
    &(v^2-u^2) \tilde S_{k,l}(v) S_{i,j}(u) -  \delta_{j,-l} (u-v)\tilde S_{k,a}(v) S_{i,-a}(u)  \\
\notag &\qquad    + \delta_{i,l} (u+v) \tilde S_{k,a}(v) S_{a,j}(u) + \delta_{j,-l} \tilde S_{k,a}(v) S_{a,-i}(u) \\
\notag     & = (v^2-u^2) S_{i,j}(u) \tilde S_{k,l}(v) -  \delta_{i,-k} (u-v) S_{-a,j}(u) \tilde S_{a,l}(v)  \\
\notag &\qquad    + \delta_{j,k} (u+v) S_{i,a}(u) \tilde S_{a,l}(v) + \delta_{i,-k} S_{-j,a}(u) \tilde S_{a,l}(v)
\end{align}
where we interpret each occurrence of $a$ to be summing over $\{-1,0,1\}$.

Next by multiplying out the righthand side of \eqref{EQ:Smatrix}, we get
\begin{equation} \label{EQ:Smatrix2}
   S(u) = \begin{pmatrix}
             D(u) & D(u) E(u) \\ F(u) D(u) & F(u) D(u) E(u) + G(u)
          \end{pmatrix}
\end{equation}
This immediately gives that $D_{i,j}(u) = S_{i,j}(u)$ for all $i,j \in \{\pm 1\}$, which implies Theorem \ref{T:1} \eqref{EQ:D}.

One can now verify that
\begin{equation} \label{EQ:Smatrix3}
   \tilde S(u) = \begin{pmatrix}
             \tilde D(u) + E(u) \tilde G(u) F(u) & -E(u) \tilde G (u) \\ -\tilde G(u) F(u)  & \tilde G(u)
          \end{pmatrix}
\end{equation}
gives an explicit formula of $\tilde S(u)$.

\subsection{Proof of Theorem \ref{T:1} \eqref{EQ:DE} and \eqref{EQ:FDI}}
\begin{proof}
Assuming that $i,j,k \in \{\pm 1\}$ and that $l=0$ in \eqref{EQ:relhelp}, we get
\begin{align*} 
    &(v^2-u^2) \tilde S_{k,0}(v) S_{i,j}(u)  \\
\notag     & = (v^2-u^2) S_{i,j}(u) \tilde S_{k,0}(v) -  \delta_{i,-k} (u-v) S_{-a,j}(u) \tilde S_{a,0}(v)  \\
\notag &\qquad    + \delta_{j,k} (u+v) S_{i,a}(u) \tilde S_{a,0}(v) + \delta_{i,-k} S_{-j,a}(u) \tilde S_{a,0}(v),
\end{align*}
which is equivalent to
\begin{align*}
    &(u^2-v^2)
[S_{i,j}(u), \tilde S_{k,0}(v)]  =  \\
&\qquad -\delta_{i,-k} (u-v) S_{-a,j}(u) \tilde S_{a,0}(v)  
    + \delta_{j,k} (u+v) S_{i,a}(u) \tilde S_{a,0}(v) + \delta_{i,-k} S_{-j,a}(u) \tilde S_{a,0}(v).
\end{align*}
Now if we only sum over $a=\pm 1$ we get
\begin{align*}
    &(u^2-v^2)
[S_{i,j}(u), \tilde S_{k,0}(v)]  =  \\
&\qquad -\delta_{i,-k} (u-v) S_{0,j}(u) \tilde S_{0,0}(v)  
    + \delta_{j,k} (u+v) S_{i,0}(u) \tilde S_{0,0}(v) + \delta_{i,-k} S_{-j,0}(u) \tilde S_{0,0}(v) \\
&\qquad -\delta_{i,-k} (u-v) S_{-a,j}(u) \tilde S_{a,0}(v)  
    + \delta_{j,k} (u+v) S_{i,a}(u) \tilde S_{a,0}(v) + \delta_{i,-k} S_{-j,a}(u) \tilde S_{a,0}(v).
\end{align*}
Using \eqref{EQ:Smatrix2} and \eqref{EQ:Smatrix3} this is equivalent to
\begin{align*}
    &(u^2-v^2)
[D_{i,j}(u), - E_{k,0}(v) \tilde G(v)]  =  \\
& \big(-\delta_{i,-k} (u-v) F_{0,a}(u)D_{a,j}(u) + \delta_{j,k} (u+v) D_{i,a}(u) E_{a,0}(u) +
\delta_{i,-k} D_{-j,a}(u) E_{a,0}(u)\big) \tilde G(v) \\
& + \big( -\delta_{i,-k} (u-v) D_{-a,j}(u) + \delta_{j,k}(u+v) D_{i,a}(u) + \delta_{i,-k} D_{-j,a}(u)\big) (-E_{a,0}(v) \tilde G(v)).
\end{align*}
Now multiplying on the right by $-G(v)$, and using the fact that $D_{i,j}(u)$ commutes with $\tilde G(v)$, gives that
\begin{align*}
    &(u^2-v^2)
[D_{i,j}(u), E_{k,0}(v) ]  =  \\
& \delta_{i,-k} (u-v) F_{0,a}(u)D_{a,j}(u) - \delta_{j,k} (u+v) D_{i,a}(u) E_{a,0}(u) -
\delta_{i,-k} D_{-j,a}(u) E_{a,0}(u) \\
& + \big( -\delta_{i,-k} (u-v) D_{-a,j}(u) + \delta_{j,k}(u+v) D_{i,a}(u) + \delta_{i,-k} D_{-j,a}(u)\big) E_{a,0}(v) .
\end{align*}

Now rearranging terms gives
\begin{align}  \label{EQ:almost}
   [D_{i,j}(u), E_{k,0}(v)] & = 
       \frac{\delta_{j,k}}{u-v} D_{i,a}(u)(E_{a,0}(v)- E_{a,0}(u)) \\
\notag       &+ \frac{\delta_{i,-k}}{u+v} ( F_{0,a}(u) D_{a,j}(u) -D_{-a,j}(u) E_{a,0}(v)) \\
\notag       &+ \frac{\delta_{i,-k}}{u^2-v^2} D_{-j,a}(u)(E_{a,0}(v)-E_{a,0}(u))
\end{align}
Now we consider the $D_{-a,j}(u) E_{a,0}(v)$ term in this equation.
By using the above commutator relation, and recalling that we are summing over $a=\pm 1$, we have that
\begin{align*}
   [D_{-a,j}(u), E_{a,0}(v)] & = 
       \frac{1}{u-v} D_{-j,a}(u)(E_{a,0}(v)- E_{a,0}(u)) \\
       &\qquad + \frac{2}{u+v} ( F_{0,a}(u) D_{a,j}(u) -D_{-a,j}(u) E_{a,0}(v)) \\
       &\qquad + \frac{2}{u^2-v^2} D_{-j,a}(u)(E_{a,0}(v)-E_{a,0}(u)) \\
       &=
       \frac{1}{u-v} D_{-j,a}(u)(E_{a,0}(v)- E_{a,0}(u)) \\
       &\qquad + \frac{2}{u+v} ( F_{0,a}(u) D_{a,j}(u) - E_{a,0}(v) D_{-a,j}(u) -[D_{-a,j}(u),E_{a,0}(v)]) \\
       &\qquad + \frac{2}{u^2-v^2} D_{-j,a}(u)(E_{a,0}(v)-E_{a,0}(u)).
\end{align*} 
Now (ignoring the middle term) we can solve this equation for $[D_{-a,j}(u),E_{a,0}(v)]$ to get
\begin{align*}
   [D_{-a,j}(u), E_{a,0}(v)] & = 
       \frac{(u+v)}{(u-v)(u+v+2)} D_{-j,a}(u)(E_{a,0}(v)- E_{a,0}(u)) \\
       &\qquad + \frac{2}{u+v+2} ( F_{0,a}(u) D_{a,j}(u) -E_{a,0}(v) D_{-a,j}(u)) \\
       &\qquad + \frac{2}{(u-v)(u+v+2)} D_{-j,a}(u)(E_{a,0}(v)-E_{a,0}(u)).
\end{align*}
Now using this 
to commute $D_{-a,j}(u)$ past $E_{a,0}(v)$ in \eqref{EQ:almost}, then simplifying gives the result.

Finally applying the anti-automorphism $\tau$ to this gives Theorem \ref{T:1} \eqref{EQ:FDI}.
\end{proof}

\subsection{Proof of Theorem \ref{T:1} \eqref{EQ:GEI} and \eqref{EQ:GFI}}
\begin{proof}
Assuming that $i \in \{\pm 1\}$ and that $j,k,l=0$ in \eqref{EQ:relhelp}, we get
\begin{align*} 
    &(v^2-u^2) \tilde S_{0,0}(v) S_{i,0}(u)   - (u-v) \tilde S_{0,a}(v) S_{i,-a}(u) + \tilde S_{0,a}(v) S_{a,-i}(u) \\
\notag     & = (v^2-u^2) S_{i,0}(u) \tilde S_{0,0}(v) +  (u+v) S_{i,a}(u) \tilde S_{a,0}(v) 
\end{align*}
which is equivalent to
\begin{align*}
    &(u^2-v^2)
[S_{i,0}(u), \tilde S_{0,0}(v)]  =  \\
&\qquad (u+v) S_{i,a}(u) \tilde S_{a,0}(v)  
    +  (u-v) \tilde S_{0,a}(v) S_{i,-a}(u) - \tilde S_{0,a}(v) S_{a,-i}(u).
\end{align*}
Now if we only sum over $a=\pm 1$ we get
\begin{align*}
    (u^2-v^2)
[S_{i,0}(u), \tilde S_{0,0}(v)]  &=  
 (u+v) (S_{i,0}(u) \tilde S_{0,0}(v) + S_{i,a}(u) \tilde S_{a,0}(v)) \\
&\quad    +  (u-v) (\tilde S_{0,0}(v) S_{i,0}(u) +  \tilde S_{0,a}(v) S_{i,-a}(u) ) \\
&\quad    - (\tilde S_{0,0}(v) S_{0,-i}(u) + \tilde S_{0,a}(v) S_{a,-i}(u)).
\end{align*}
Using \eqref{EQ:Smatrix2} and \eqref{EQ:Smatrix3} this is equivalent to
\begin{align*}
    (u^2-v^2)
[D_{i,a}(u) E_{a,0}(u), \tilde G (v)]  &=  
 (u+v) (D_{i,a}(u)E_{a,0}(u) \tilde G(v) - D_{i,a}(u) E_{a,0}(v) \tilde G(v)) \\
&\quad    +  (u-v) (\tilde G(v) D_{i,a}(u)E_{a,0}(u)  -  \tilde G(v) F_{0,-a}(v) D_{i,a}(u) ) \\
&\quad    - (\tilde G(v) F_{0,a}(u)D_{a,-i}(u) - \tilde G(v) F_{0,a}(v) D_{a,-i}(u)).
\end{align*}

Now by Theorem \ref{T:1} \eqref{EQ:FDI} (and the fact that we are summing over $a \in \{\pm 1\}$) we have that
\begin{align*}
[F_{0,-a}(v),D_{i,a}(u)] &= \frac{1}{u-v} ( F_{0,a}(v) - F_{0,a}(u)) D_{a,-i}(u) \\
&\quad + \frac{2}{u+v+2} D_{i,a}(u)(E_{a,0}(u) - F_{0,-a}(v)).
\end{align*}
Using this to commute $F_{0,-a}(v)$ past $D_{i,a}(u)$ in the previous equation gives that 
\begin{align*}
    (u^2-v^2)
[D_{i,a}(u) E_{a,0}(u), \tilde G (v)]  &=  
 (u+v) (D_{i,a}(u)E_{a,0}(u) \tilde G(v) - D_{i,a}(u) E_{a,0}(v) \tilde G(v)) \\
&\quad    +  (u-v) (\tilde G(v) D_{i,a}(u)E_{a,0}(u)  -  \tilde G(v) D_{i,a}(u) F_{0,-a}(v) ) \\
&\qquad - \frac{2(u-v)}{u+v+2} \tilde G(v) D_{i,a}(u) (E_{a,0}(u) - F_{0,-a}(v)).
\end{align*}
Now we can consider this to be an equation of column vectors by allowing $i$ to vary over $\pm 1$.  Multiplying this on the left by $\tilde D(u)$, and using the fact
that $D(u)$ and $G(v)$ commute, gives 
\begin{align*}
    (u^2-v^2)
[E_{i,0}(u), \tilde G (v)]  &=  
 (u+v) (E_{i,0}(u) \tilde G(v) - E_{i,0}(v) \tilde G(v)) \\
&\quad    +  (u-v) (\tilde G(v) E_{i,0}(u)  -  \tilde G(v) F_{0,-i}(v) ) \\
&\qquad - \frac{2(u-v)}{u+v+2} \tilde G(v) (E_{i,0}(u) - F_{0,-i}(v)),
\end{align*}
which is easily seen to be equivalent to Theorem \ref{T:1} \eqref{EQ:GEI}.

Finally applying the anti-automorphism $\tau$ to this gives Theorem \ref{T:1} \eqref{EQ:GFI}.
\end{proof}

\subsection{Proof of Theorem \ref{T:1} \eqref{EQ:EFI}}
\begin{proof}
Assuming that $i,l \in \{\pm 1\}$ and that $j,k=0$ in \eqref{EQ:relhelp}, we get
\begin{align*} 
    &(v^2-u^2) \tilde S_{0,l}(v) S_{i,0}(u)   + \delta_{i,l} (u+v) \tilde S_{0,a}(v) S_{a,0}(u)  \\
\notag     & = (v^2-u^2) S_{i,0}(u) \tilde S_{0,l}(v) +  (u+v) S_{i,a}(u) \tilde S_{a,l}(v) 
\end{align*}
which is equivalent to
\begin{align*}
    &(u^2-v^2)
[S_{i,0}(u), \tilde S_{0,l}(v)]  = 
 -\delta_{i,l} (u+v) \tilde S_{0,a}(v) S_{a,0}(u)  
    +  (u+v) S_{i,a}(u) \tilde S_{a,l}(v).
\end{align*}
Now if we only sum over $a=\pm 1$ we get
\begin{align*}
    (u^2-v^2)
[S_{i,0}(u), \tilde S_{0,l}(v)]  &= 
 -\delta_{i,l} (u+v) \tilde S_{0,0}(v) S_{0,0}(u)  
 -\delta_{i,l} (u+v) \tilde S_{0,a}(v) S_{a,0}(u)  \\
 &\quad
    +  (u+v) S_{i,0}(u) \tilde S_{0,l}(v)
    +  (u+v) S_{i,a}(u) \tilde S_{a,l}(v).
\end{align*}
Using \eqref{EQ:Smatrix2} and \eqref{EQ:Smatrix3} this is equivalent to
\begin{align} \label{EQ:reln1}
    &(u-v)
[D_{i,a}(u) E_{a,0}(u), -\tilde G(v) F_{0,l}(v)]  =  \\ \notag
&\qquad - \delta_{i,l}\tilde G(v)(F_{0,a}(u) D_{a,b}(u) E_{b,0}(u) + G(u)) + \delta_{i,l} \tilde G(v) F_{0,a}(v) D_{a,b}(u) E_{b,0}(u) \\ \notag
&\qquad - D_{i,a}(u) E_{a,0}(u) \tilde G(v) F_{0,l}(v) + D_{i,a}(u) (\tilde D_{a,l}(v) + E_{a,0}(v) \tilde G(v) F_{0,l}(v)).
\end{align}
Now the left hand side of \eqref{EQ:reln1} is 
\[
  (u-v) (\tilde G(v) F_{0,l}(v) D_{i,a}(u) E_{a,0}(u) - D_{i,a}(u) E_{a,0}(u) \tilde G(v) F_{0,l}(v)).
\]
If we use the commutator formulas for $[F_{0,l}(v),D_{i,a}(u)]$ and $[E_{a,0}(u) \tilde G(v)]$ from Theorem \ref{T:1}, we get that this equals
\begin{align*}
   (u-v) &\Big( 
    \tilde G(v) D_{i,a}(u) F_{0,l}(v) E_{a,0}(u) - D_{i,a}(u) \tilde G(v) E_{a,0}(u) F_{0,l}(v)  \\
     & \qquad +\frac{\delta_{i,l}}{u-v} \tilde G(v) (F_{0,b}(v) - F_{0,b}(u))D_{b,a}(u) E_{a,0}(u) \\ 
    &\qquad + \frac{1}{u+v+2} \tilde G(v) D_{i,a}(u) (E_{a,0}(u)-F_{0,-a}(v)) E_{-l,0}(u) \\
    &\qquad -\frac{1}{u-v} D_{i,a}(u) (E_{a,0}(u) - E_{a,0}(v)) \tilde G(v) F_{0,l}(v) \\
    &\qquad - \frac{1}{u+v+2} D_{i,a}(u) \tilde G(v) (E_{a,0}(u) - F_{0,-a}(v)) F_{0,l}(v)
   \Big).
\end{align*}
Note that the second and fourth lines of this cancel with terms from the right hand side of \eqref{EQ:reln1}. 
Substituting this into \eqref{EQ:reln1} and canceling terms now yields that 
\begin{align*}
   (u-v) \tilde G(v) D_{i,a}(u) [ F_{0,l}(v),E_{a,0}(u)] &= -\delta_{i,l} \tilde G(v) G(u) + D_{i,a}(u) \tilde D_{a,l}(v) \\
    &\quad - \frac{u-v}{u+v+2} \tilde G(v) D_{i,a} (u) (E_{a,0}(u) - F_{0,-a}(v)) E_{-l,0}(u) \\
    &\quad + \frac{u-v}{u+v+2} D_{i,a}(u) \tilde G(v)(E_{a,0}(u) - F_{0,-a}(v)) F_{0,l}(v).
\end{align*}
We can treat this as an equation of column vectors (letting $i$ run over $\pm 1$).  Now if we multiply this on the left by $G(v) \tilde D(u)$,
we get
\begin{align*}
   (u-v) [ F_{0,l}(v),E_{i,0}(u)] &= -\tilde D_{i,l}(u) G(u) + G(v) \tilde D_{i,l}(v) \\
    &\quad - \frac{u-v}{u+v+2} (E_{i,0}(u) - F_{0,-i}(v)) E_{-l,0}(u) \\
    &\quad + \frac{u-v}{u+v+2} (E_{i,0}(u) - F_{0,-i}(v)) F_{0,l}(v),
\end{align*}
which simplifies to
\begin{align*}
   [ E_{i,0}(u),F_{0,l}(v)] &= \frac{1}{u-v} (\tilde D_{i,l}(u) G(u) - G(v) \tilde D_{i,l}(v)) \\
    &\quad + \frac{1}{u+v+2} (E_{i,0}(u) - F_{0,-i}(v)) (E_{-l,0}(u)-F_{0,l}(v)).
\end{align*}
\end{proof}

\subsection{Proof of Theorem \ref{T:1} \eqref{EQ:EEI} and \eqref{EQ:FFI}}
\begin{proof}
Assuming that $i,k \in \{\pm 1\}$ and that $j,l=0$ in \eqref{EQ:relhelp}, we get
\begin{align*} 
    &(v^2-u^2) \tilde S_{k,0}(v) S_{i,0}(u)   - (u-v) \tilde S_{k,a}(v) S_{i,-a}(u) + \tilde S_{k,a}(v) S_{a,-i}(u) \\
\notag     & = (v^2-u^2) S_{i,0}(u) \tilde S_{k,0}(v)   - \delta_{i,-k} S_{-a,0}(u) \tilde S_{a,0}(v) + \delta_{i,-k} S_{0,a}(u) \tilde S_{a,0}(v)
\end{align*}
which is equivalent to
\begin{align*}
    (u^2-v^2)
[S_{i,0}(u), \tilde S_{k,0}(v)]  &= 
(u-v)(\tilde S_{k,a}(v) S_{i,-a}(u) - \delta_{i,-k} S_{-a,0}(u) \tilde S_{a,0}(v)) \\
&\quad - \tilde S_{k,a}(v) S_{a,-i}(u) + \delta_{i,-k} S_{0,a}(u) \tilde S_{a,0}(v).
\end{align*}
Now if we only sum over $a=\pm 1$ we get
\begin{align*}
    (u^2-v^2)
[S_{i,0}(u), \tilde S_{k,0}(v)]  &= 
(u-v)\big(\tilde S_{k,0}(v) S_{i,0}(u) +
\tilde S_{k,a}(v) S_{i,-a}(u)  \\ 
&\quad - \delta_{i,-k} S_{0,0}(u) \tilde S_{0,0}(v) 
 - \delta_{i,-k} S_{-a,0}(u) \tilde S_{a,0}(v) \big) \\
&\quad - \tilde S_{k,0}(v) S_{0,-i}(u) 
- \tilde S_{k,a}(v) S_{a,-i}(u) \\ 
&\quad +\delta_{i,-k} S_{0,0}(u) \tilde S_{0,0}(v)
+\delta_{i,-k} S_{0,a}(u) \tilde S_{a,0}(v).
\end{align*}
Using \eqref{EQ:Smatrix2} and \eqref{EQ:Smatrix3} this is equivalent to
\begin{align}  \label{EQ:EGcomm1}
    &-(u^2-v^2)
    [D_{i,a}(u) E_{a,0}(u), E_{k,0}(v) \tilde G(v)] =  \\ \notag
 &    (u-v) \big ( -E_{k,0}(v) \tilde G(v) D_{i,a}(u) E_{a,0}(u) + (\tilde D_{k,a}(v) +E_{k,0}(v) \tilde G(v) F_{0,a}(v))D_{i,-a}(u) \\ \notag
 &   - \delta_{i,-k} (F_{0,a}(u) D_{a,b}(u) E_{b,0}(u) + G(u)) \tilde G(v) + \delta_{i,-k} D_{-a,b}(u) E_{b,0}(u) E_{a,0}(v) \tilde G(v) \big) \\ \notag
 &  + E_{k,0}(v) \tilde G(v) F_{0,a}(u) D_{a,-i}(u) - (\tilde D_{k,a}(v) +E_{k,0}(v) \tilde G(v) F_{0,a}(v)) D_{a,-i}(u) \\ \notag
 &  \delta_{i,-k} (F_{0,a}(u) D_{a,b}(u) E_{b,0}(u) + G(u)) \tilde G(v) - \delta_{i,-k} F_{0,b}(u) D_{b,a}(u) E_{a,0}(v) \tilde G(v).
\end{align}
Now the left hand side of this equation is
\begin{align} \label{EQ:EGcomm}
&(u^2-v^2) (E_{k,0}(v) \tilde G(v) D_{i,a}(u) E_{a,0}(u) - D_{i,a}(u) E_{a,0}(u) E_{k,0}(v) \tilde G(v)) = \\
&\quad (u^2-v^2) \big(D_{i,a}(u) E_{k,0}(v) E_{a,0}(u) \tilde G(v) - D_{i,a}(u) E_{a,0}(u) E_{k,0}(v) \tilde G(v)  \notag \\ 
&\quad 
+ E_{k,0}(v) D_{i,a}(u)[\tilde G(v), E_{a,0}(u)] 
+ [E_{k,0}(v),D_{i,a}(u)] E_{a,0}(u) \tilde G(v) \big).\notag
\end{align}
Note that 
\begin{align*}
E_{k,0}(v) \tilde G(v) D_{i,a}(u) E_{a,0}(u) & = 
E_{k,0}(v) D_{i,a}(u) \tilde G(v) E_{a,0}(u) \\
&= D_{i,a}(u) E_{k,0}(v) \tilde G(v) E_{a,0}(u) - [E_{k,0}(v) ,D_{i,a}(u)] \tilde G(v) E_{a,0}(u) \\
 & =  D_{i,a}(u) E_{k,0}(v) E_{a,0}(u) \tilde G(v) -D_{i,a}(u) E_{k,0}(v) [ E_{a,0}(u)\tilde G(v)]  \\
 &\qquad - [E_{k,0}(v) ,D_{i,a}(u)] \tilde G(v) E_{a,0}(u)
\end{align*}

Now we use this and the fact that
\begin{align*}
   [E_{a,0}(u), \tilde G(v)] &= \frac{1}{u-v} (E_{a,0}(u) - E_{a,0}(v)) \tilde G(v) \\
    & + \frac{1}{u+v}(\tilde G(v) E_{a,0}(u) - \tilde D_{a,b}(u) \tilde G(v) F_{0,-c}(v) D_{b,c}(u)) \\
    & + \frac{1}{u^2-v^2} \tilde G(v) \tilde D_{a,b}(u) (F_{0,c}(v) - F_{0,c}(u))D_{c,-b}(u).
\end{align*}
and
\begin{align} \label{EQ:ERDP}
  [E_{k,0}(v), D_{i,a}(u)] &= \frac{\delta_{a,k}}{u-v} D_{i,b}(u) (E_{b,0}(u) - E_{b,0}(v)) \\ \notag
  & + \frac{\delta_{i,-k}}{u+v} (-F_{0,b}(u) D_{b,a}(u) + D_{-b,a}(u) E_{b,0}(v)) \\ \notag
  & + \frac{\delta_{i,-k}}{u^2-v^2} D_{-a,b}(u) (E_{b,0}(u) - E_{b,0}(v))
\end{align}
to get that 
the left hand side of \eqref{EQ:EGcomm} is
\begin{align} \label{EQ:EGcomm2}
&(u^2-v^2) \big( D_{i,a}(u) E_{k,0}(v) E_{a,0}(u) \tilde G(v) - D_{i,a}(u) E_{a,0}(u) E_{k,0}(v) \tilde G(v) \big) \\ \notag
&\quad +(u+v) (E_{k,0}(v) D_{i,a}(u)(E_{a,0}(u) - E_{a,0}(v)) \tilde G(v)) \\ \notag
&\quad + (u-v)  E_{k,0}(v) D_{i,a}(u) \big(\tilde D_{a,b}(u) \tilde G(v) F_{0,-c}(v) D_{b,c}(u) - \tilde G(v) E_{a,0}(u) \big) \\ \notag
&\quad + E_{k,0}(v) D_{i,a}(u) \tilde G(v) \tilde D_{a,b}(u) \big( F_{0,c}(u) - F_{0,c}(v) \big) D_{c,b}(u) \\ \notag
&\quad + (u+v) D_{i,b}(u) \big(E_{b,0}(u) - E_{b,0}(v) \big) E_{k,0}(u) \tilde G(v) \\ \notag
&\quad + (u-v) \delta_{i,-k} \big( D_{-b,a}(u) E_{b,0}(v) - F_{0,b}(u) D_{b,a}(u) \big) E_{a,0}(u) \tilde G(v) \\ \notag
&\quad + \delta_{i,-k} D_{-a,b}(u) \big( E_{b,0}(u) - E_{b,0}(v)\big) E_{a,0}(u) \tilde G(v).
\end{align}

Now we use \eqref{EQ:ERDP} on the second line of this to get that
\begin{align*}
 &(u+v) (E_{k,0}(v) D_{i,a}(u)(E_{a,0}(u) - E_{a,0}(v)) \tilde G(v))  = \\ 
&\qquad  (u+v) (D_{i,a}(u) E_{k,0}(v)(E_{a,0}(u) - E_{a,0}(v)) \tilde G(v)) \\
 & \qquad  + \frac{u+v}{u-v} D_{i,b}(u)(E_{b,0}(u) - E_{b,0}(v))(E_{k,0}(v) - E_{k,0}(u)) \tilde G(v) \\
 &\qquad  + \delta_{i,-k}(D_{-b,a}(u) E_{b,0}(v) - F_{0,b}(u) D_{b,a}(u))(E_{a,0}(v)-E_{a,0}(u)) \tilde G(v) \\
 &\qquad + \frac{\delta_{i,-k}}{u-v} D_{-a,b}(u)(E_{b,0}(u)-E_{b,0}(v))(E_{a,0}(v)-E_{a,0}(u)) \tilde G(v).
\end{align*}

We substitute this into \eqref{EQ:EGcomm2} and then cancel like terms in \eqref{EQ:EGcomm1}.  Furthermore we
consider \eqref{EQ:EGcomm1} to be an equation of column vectors where $i \in \{\pm1\}$, and multiply both sides on the left by $\tilde D(u)$ and on the right by $G(v)$.  This results in
\begin{align*}
  &(u^2-v^2)[E_{k,0}(v),E_{i,0}(u)] + (u+v) E_{k,0}(v)(E_{i,0}(v) - E_{i,0}(u)) \\
  &\qquad + \frac{u+v}{u-v}(E_{i,0}(u) - E_{i,0}(v))(E_{k,0}(v) - E_{k,0}(u))  \\
  &\qquad + \tilde D_{i,-k}(u) D_{-b,a}(u) E_{b,0}(v)(E_{a,0}(v) - E_{a,0}(u))\\
  &\qquad + \frac{1}{u-v} \tilde D_{i,-k}(u) D_{-a,b}(u)(E_{b,0}(u) - E_{b,0}(v))(E_{a,0}(v)-E_{a,0}(u)) \\
  &\qquad + (u+v)(E_{i,0}(u) - E_{i,0}(v)) E_{k,0}(u) + (u-v) \tilde D_{i,-k}(u) D_{-b,a}(u) E_{b,0}(v) E_{a,0}(u) \\
  &\qquad + \tilde D_{i,-k}(u) D_{-a,b}(u)(E_{b,0}(u)-E_{b,0}(v))E_{a,0}(u) \\
  &\quad = (u-v) \tilde D_{i,b}(u) \tilde D_{k,a}(v) D_{b,-a}(u) G(v) - (u-v) \tilde D_{i,-k}(u) G(u) \\
  &\qquad - (u-v) \tilde D_{i,-k}(u) D_{-a,b}(u) E_{b,0}(u) E_{a,0}(v) \\
  &\qquad - \tilde D_{i,b}(u) \tilde D_{k,a}(v) D_{a,-b}(u) G(v) + \tilde D_{i,-k}(u)G(u).
\end{align*}
Now we bring all but the first term on the left hand side of this over to the right hand side and we get
\begin{align} \label{EQ:EGcomm3}
 (u^2-v^2) [E_{k,0}(v), E_{i,0}(u)] &= 
 (u+v) \big(E_{k,0}(v) (E_{i,0}(u) - E_{i,0}(v)) + (E_{i,0}(v) - E_{i,0}(u)) E_{k,0}(u) \big) \\ \notag
 & \quad + (u-v) \big(\tilde D_{i,-k}(u) D_{-a,b}(u) [E_{b,0}(u),E_{a,0}(v)] \\ \notag
 &\qquad + \tilde D_{i,b}(u) \tilde D_{k,a}(v) D_{b,-a}(u) G(v) - \tilde D_{i,-k}(u) G(u) \big) \\ \notag
 &\quad + \frac{u+v}{u-v} (E_{i,0}(u) - E_{i,0}(v))(E_{k,0}(u)-E_{k,0}(v)) \\ \notag
 &\quad  + \tilde D_{i,-k}(u) D_{-b,a}(u) E_{b,0}(v)(E_{a,0}(u) - E_{a,0}(v)) \\ \notag
 &\quad+ \tilde D_{i,-r}(u) D_{-a,b}(u) (E_{b,0}(v) - E_{b,0}(u)) E_{a,0}(u) \\ \notag
 &\quad - \tilde D_{i,b}(u) \tilde D_{k,a}(v) D_{a,-b}(u) G(v) + \tilde D_{i,-k}(u) G(u) \\
 &\quad + \frac{1}{u-v} \tilde D_{i,-k}(u) D_{-a,b}(u) (E_{b,0}(u) - E_{b,0}(v))(E_{a,0}(u)-E_{a,0}(v)).
\end{align}

In order to deal with 
$\tilde D_{i,-k}(u) D_{-a,b}(u) [E_{b,0}(u),E_{a,0}(v)]$,
we let
\begin{align*}
  X &= -
 \tilde D_{i,-k}(u) D_{-a,b}(u) [E_{b,0}(u),E_{a,0}(v)] \\
 &\quad + \frac{1}{u-v} \big(\tilde D_{i,-k}(u) D_{-b,a}(u) E_{b,0}(v) (E_{a,0}(v)- E_{a,0}(u)) \\
  &\qquad + \tilde D_{i,-k}(u) D_{-a,b}(u) (E_{b,0}(u) - E_{b,0}(v)) E_{a,0}(u)\big) \\
  &\quad + \frac{1}{(u-v)^2} \tilde D_{i,-k}(u) D_{-a,b}(u) (E_{b,0}(v) - E_{b,0}(u))(E_{a,0}(u) - E_{a,0}(v)).
\end{align*}
Now, omitting some calculations, we can use \eqref{EQ:EGcomm3} in this formula to get $X$ to show up on the right hand side.  Then solving for $X$,
we get
\begin{align*}
   X &= \frac{1}{u+v+2}(\tilde D_{i,-k}(u) \tilde D_{-b,a}(v) D_{b,-a}(u) G(v) -2 \tilde D_{i,-k}(u) G(u) ) \\
   &\quad -\frac{1}{(u-v)(u+v+2)} (\tilde D_{i,-k}(u) \tilde D_{-b,a}(v) D_{a,-b}(u) G(v) -2  \tilde D_{i,-k}(u) G(u)).
\end{align*}

Substituting this into \eqref{EQ:EGcomm3}, dividing both sides by $-(u^2-v^2)$, and simplifying  gives
\begin{align} \label{EQ:EGcomm4}
  [E_{i,0}(u),E_{q,0}(v)] &= 
 \frac{1}{u-v} \big(E_{k,0}(v) (E_{i,0}(v) - E_{i,0}(u)) + (E_{i,0}(u) - E_{i,0}(v)) E_{k,0}(u) \big) \\ \notag 
  &\quad - \frac{1}{u+v} \tilde D_{i,b}(u) \tilde D_{k,-a}(v) D_{b,-a}(u) G(v)  \\ \notag
  &\quad + \frac{1}{(u-v)^2} (E_{i,0}(v)-E_{i,0}(u))(E_{k,0}(u)-E_{k,0}(v))) \\ \notag
  &\quad + \frac{1}{u^2-v^2}\tilde D_{i,b}(u) \tilde D_{k,a}(v) D_{a,-b}(u) G(v) \\ \notag
  &\quad + \frac{1}{(u+v+2)(u+v)} \tilde D_{i,-k}(u) \tilde D_{-b,a}(v) D_{b,-a}(u) G(v)  \\ \notag
  &\quad - \frac{1}{(u^2-v^2)(u+v+2)} \tilde D_{i,-k}(u) \tilde D_{-b,a}(v) D_{a,-b}(u) G(v) \\ \notag
  &\quad + \left( \frac{1}{u+v+2} - \frac{1}{(u-v)(u+v+2)} \right ) \tilde D_{i,-k}(u) G(u).
\end{align}

Now we use Theorem \ref{T:1} \eqref{EQ:D} to get get that
\begin{align*}
 [D_{b,-a}(u), \tilde D_{k,a}(v)] &= \frac{1}{u+v} \big(2 \tilde D_{k,c}(v) D_{b,-c}(u) - \delta_{k,-b} D_{-c,-a}(u) \tilde D_{c,a}(v)\big) \\
 &\quad + \frac{1}{u-v}\big(-\delta_{a,b} \tilde D_{k,c}(v) D_{c,-a}(u) + D_{b,c}(u) \tilde D_{c,-k}(v)\big) \\
 &\quad + \frac{1}{u^2-v^2} \big(-2 \tilde D_{k,c}(v) D_{c,-b}(u) + \delta_{k,-b} D_{a,c}(u) \tilde D_{c,a}(v)\big).
\end{align*}

So if we let $X= \tilde D_{i,b}(u) \tilde D_{k,a}(v) D_{b,-a}(u)G(v) - \frac{1}{u-v} \tilde D_{i,b}(u) \tilde D_{k,a}(v) D_{a,-b}(u)G(v)$ and use the above commutator
formula $X$ will show up on the right hand side of this equation. If we solve for $X$ we get
\begin{align} \label{EQ:X}
 X  &= \frac{u+v}{u+v+2} \tilde D_{k,-i}(v)G(v) + \frac{1}{u+v+2} \tilde D_{i,-k}(u) D_{-b,-a}(u) \tilde D_{b,a}(v) G(v) \\ \notag
&\quad - \frac{u+v}{(u+v+2)(u-v)} \tilde D_{i,-k}(v) G(v) - \frac{1}{(u+v+2)(u-v)} \tilde D_{i,-k}(u) D_{a,b}(u) \tilde D_{b,a}(v) G(v).
\end{align}

We also use Theorem \ref{T:1} \eqref{EQ:D} to get
\begin{align*}
 [D_{b,a}(u), \tilde D_{-b,a}(v)] &= \frac{2}{u+v} \big(\tilde D_{-b,c}(v) D_{b,-c}(u) - D_{-c,-b}(u) \tilde D_{c,b}(v) \big) \\
 &\quad  + \frac{1}{u-v}\big(- \tilde D_{-b,c}(v) D_{c,-b}(u) + D_{b,c}(u) \tilde D_{c,b}(v) \big) \\
 &\quad + \frac{2}{u^2-v^2} \big( -\tilde D_{-b,c}(v) D_{c,-b}(u) + D_{b,c}(u) \tilde D_{c,b}(v)\big) \\
 & = \frac{1}{u-v} [D_{a,b}(u),\tilde D_{b,a}(v)].
\end{align*}

Using this and \eqref{EQ:X} in \eqref{EQ:EGcomm4} gives
\begin{align*}
  [E_{i,0}(u),E_{k,0}(v)] &= 
 \frac{1}{u-v} \big(E_{k,0}(v) (E_{i,0}(v) - E_{i,0}(u)) + (E_{i,0}(u) - E_{i,0}(v)) E_{k,0}(u) \big) \\ \notag 
  &\quad + \frac{1}{(u-v)^2} (E_{i,0}(v)-E_{i,0}(u))(E_{k,0}(u)-E_{k,0}(v))) \\ \notag
  &\quad - \frac{1}{u+v+2} \tilde D_{k,-i}(v) G(v) \\
  &\quad + \frac{1}{(u+v+2)(u-v)} \tilde D_{i,-k}(v) G(v) \\
  &\quad + \left( \frac{1}{u+v+2} - \frac{1}{(u-v)(u+v+2)} \right ) \tilde D_{i,-k}(u) G(u).
\end{align*}

Now if we commute $E_{k,0}(v)$ to the right of $E_{i,0}(v)$ and $E_{i,0}(u)$ and solve for $[E_{i,0}(u),E_{k,0}(v)]$ we get
\begin{align*}
  \frac{u-v-1}{u-v} [E_{i,0}(u),E_{k,0}(v)] &= -\frac{1}{u-v} [E_{i,0}(v),E_{k,0}(v)]  \\
  &\quad + \frac{1}{(u-v)} (E_{i,0}(u)-E_{i,0}(v))(E_{k,0}(u)-E_{k,0}(v))) \\ \notag
  &\quad - \frac{1}{(u-v)^2} (E_{i,0}(u)-E_{i,0}(v))(E_{k,0}(u)-E_{k,0}(v))) \\ \notag
  &\quad - \frac{1}{u+v+2} \tilde D_{k,-i}(v) G(v) \\
  &\quad + \frac{1}{(u+v+2)(u-v)} \tilde D_{i,-k}(v) G(v) \\
  &\quad + \left( \frac{1}{u+v+2} - \frac{1}{(u-v)(u+v+2)} \right ) \tilde D_{i,-k}(u) G(u).
\end{align*}

If we set $u=v+1$ in this, we get
\[
[E_{i,0}(v),E_{k,0}(v)] = -\frac{1}{2v+3} \tilde D_{k,-i}(v) G(v) + \frac{1}{2v+3} \tilde D_{i,-k}(v)G(v).
\]
Plugging this into the previous equation now gives
\begin{align*}
  [E_{i,0}(u),E_{k,0}(v)] &= \frac{1}{(u-v)} (E_{i,0}(u)-E_{i,0}(v))(E_{k,0}(u)-E_{k,0}(v))) \\ \notag
  &\quad + \frac{1}{u+v+2}  \tilde D_{i,-k}(u) G(u) \\
  &\quad - \frac{2(v+1)}{(2v+3)(u+v+2)} \tilde D_{k,-i}(v) G(v) \\
  &\quad - \frac{1}{(2v+3)(u+v+2)} \tilde D_{i,-k}(v) G(v).
\end{align*}

Finally applying the anti-automorphism $\tau$ to this gives Theorem \ref{T:1} \eqref{EQ:FFI}.

\end{proof}

\subsection{Proof of Theorem \ref{T:1} \eqref{EQ:GG}}
\begin{proof}
We will prove that $[\tilde G(v), G(u)] = 0$, which implies Theorem \ref{T:1} \eqref{EQ:GG}.

By \eqref{EQ:Smatrix2} $S_{0,0}(u) = G(u) + F_{0,a}(u) D_{a,b}(u) E_{b,0}(u)$, and by \eqref{EQ:Smatrix3} $\tilde S_{0,0}(v) = \tilde G(v)$.
Thus 
\[
[\tilde S_{0,0}(v), S_{0,0}(u)] = [\tilde G(v), G(u)] + [\tilde G(v), F_{0,a}(u) D_{a,b}(u) E_{b,0}(u)].  
\]
Therefore 
it suffices to show that
\[
[\tilde G(v), F_{0,a}(u) D_{a,b}(u) E_{b,0}(u)] - 
[\tilde S_{0,0}(v), S_{0,0}(u)] = 0.
\]

By \eqref{EQ:relhelp}, where $i,j,k,l = 0$, we have that
\begin{align*}
   (v^2&-u^2) 
   \tilde S_{0,0}(v), S_{0,0}(u)
   -(u-v) 
   \tilde S_{0,a}(v) S_{0,-a}(u) 
   + (u+v) 
   \tilde S_{0,a}(v) S_{a,0}(u) 
   + \tilde S_{0,a}(v) S_{a,0}(u)  = \\
   &
   (v^2-u^2) S_{0,0}(u), \tilde S_{0,0}(v) 
   -(u-v) 
   S_{-a,0}(u) \tilde S_{a,0}(v) 
   +(u+v) 
   S_{0,a}(u) \tilde S_{a,0}(v) 
   +
   S_{0,a}(u) \tilde S_{a,0}(v),
\end{align*} 
where we are summing over $a \in \{-1, 0, 1\}$.
If we instead only sum over $a \in \{\pm 1\}$ we get
\begin{align*}
   (v^2&-u^2) 
   \tilde S_{0,0}(v), S_{0,0}(u)
   -(u-v) 
   \left(\tilde S_{0,a}(v) S_{0,-a}(u)  +
   \tilde S_{0,0}(v) S_{0,0}(u) \right) \\
   &\quad + (u+v) 
   \left(\tilde S_{0,a}(v) S_{a,0}(u) +
   \tilde S_{0,0}(v) S_{0,0}(u) \right) 
   + \tilde S_{0,a}(v) S_{a,0}(u)  +
   \tilde S_{0,0}(v) S_{0,0}(u)  = \\
   &
   (v^2-u^2) S_{0,0}(u), \tilde S_{0,0}(v) 
   -(u-v) 
   \left(S_{-a,0}(u) \tilde S_{a,0}(v) 
   +S_{0,0}(u) \tilde S_{0,0}(v) \right) \\
   &\quad +(u+v) 
   \left(S_{0,a}(u) \tilde S_{a,0}(v)  +
   S_{0,0}(u) \tilde S_{0,0}(v) \right)
   +
   S_{0,a}(u) \tilde S_{a,0}(v) + 
   S_{0,0}(u) \tilde S_{0,0}(v).
\end{align*} 
Now by solving this equation for $[\tilde S_{0,0}(v), S_{0,0}(u)]$ we get
\begin{align} \label{EQ:s0comm}
[\tilde S_{0,0}(v), S_{0,0}(u)] &= 
-\tfrac{u-v}{(u+v+1)(u-v-1)}
\tilde S_{0,a}(v) S_{0,-a}(u) +
\tfrac{1}{u-v-1}
\tilde S_{0,a}(v) S_{a,0}(u) \\ \notag
&\quad 
+\tfrac{u-v}{(u+v+1)(u-v-1)}
S_{-a,0}(u) \tilde S_{a,0}(v)  -
\tfrac{1}{u-v-1}
S_{0,a}(u) \tilde S_{a,0}(v).
\end{align}

Next we turn our attention to 
\begin{align}  \label{EQ:GFDE}
[\tilde G(v), F_{0,a}(u) D_{a,b}(u) E_{b,0}(u)] & = 
\tilde G(v) F_{0,a}(u) D_{a,b}(u) E_{b,0}(u) - F_{0,a}(u) D_{a,b}(u) E_{b,0}(u)\tilde G(v).
\end{align}

Note that by multiplying both sides of Theorem \ref{T:1} \eqref{EQ:GEI} by $\tilde G(v)$ gives
\begin{align*}
     [E_{b,0}(v), \tilde G(u)] &=  \tfrac{1}{u-v}(E_{b,0}(u) - E_{b,0}(v)) \tilde G(u) \\ \notag
     &\quad + \tfrac{1}{u+v+2} \tilde G(u) (E_{b,0}(v) -F_{0,-b}(u)).
     \end{align*}
Now swapping $u$'s and $v$'s gives
\begin{align*}
     [E_{b,0}(u), \tilde G(v)] &=  \tfrac{1}{u-v}(E_{b,0}(u) - E_{b,0}(v)) \tilde G(v) \\ \notag
     &\quad + \tfrac{1}{u+v+2} \tilde G(v) (E_{b,0}(u) -F_{0,-b}(v)).
\end{align*}
     Now solving this for $E_{b,0}(u) \tilde G(v)$ gives
\begin{align} \label{EQ:EGuseful}
     E_{b,0}(u) \tilde G(v) &=  \tfrac{(u-v)(u+v+3)}{(u-v-1)(u+v+2)} \tilde G(v) E_{b,0}(u) -\tfrac{1}{u-v-1} E_{b,0}(v) \tilde G(v) \\ \notag
     &\quad - \tfrac{u-v}{(u+v+2)(u-v-1)} \tilde G(v) F_{0,-b}(v).
\end{align}
Also applying the anti-automorphism $\tau$ to this, and replacing $-b$ with $a$ gives
\begin{align} \label{EQ:GFuseful}
     \tilde G(v) F_{0,a}(u)  &=  \tfrac{(u-v)(u+v+3)}{(u-v-1)(u+v+2)}  F_{0,a}(u) \tilde G(v) -\tfrac{1}{u-v-1} \tilde G(v) F_{0,a}(v))  \\ \notag
     &\quad - \tfrac{u-v}{(u+v+2)(u-v-1)}  E_{-a,0}(v) \tilde G(v).
\end{align}
Now 
using the fact that $D_{a,b}$ and $\tilde G(v)$  commute, when
we plug these formulas into \eqref{EQ:GFDE} and simplify we get
\begin{align} \label{EQ:Step 1}
[\tilde G(v), &F_{0,a}(u) D_{a,b}(u) E_{b,0}(u)]  = 
\tfrac{1}{u-v-1} \left(F_a(u) D_{a,b}(u) E_b (v) \tilde G(v) - \tilde G(v) F_a(v) D_{a,b}(u) E_b(u)\right) \\ \notag
&\qquad \qquad \qquad 
+ \tfrac{u-v}{(u+v+2)(u-v-1)} \left( F_a(u) D_{a,b}(u) \tilde G(v) F_{-b}(v)   - 
E_{-a}(v) \tilde G(v) D_{a,b}(u) E_b(u)\right) \\ \notag
& \qquad \qquad \qquad \qquad \quad  
\phantom{:}
\phantom{:}
=  \tfrac{1}{u-v-1} \left (\tilde S_{0,a}(v) S_{a,0}(u) - S_{0,a}(u) \tilde S_{a,0}(v) \right ) \\ \notag
& \qquad \qquad \qquad \qquad \quad  
\phantom{:}
\phantom{:}
\quad + \tfrac{u-v}{(u+v+2)(u-v-1)} \left( \tilde S_{-a,0}(v) S_{a,0}(u) - S_{0,a}(u) \tilde S_{0, -a}(v)\right).
\end{align}
Now combining this in \eqref{EQ:s0comm} (and canceling like terms) gives
\begin{align*}
[\tilde G(v), &F_{0,a}(u) D_{a,b}(u) E_{b,0}(u)] -[\tilde S_{0,0}(v), S_{0,0}(u)]  =  \\
& \tfrac{u-v}{(u+v+2)(u-v-1)} \left( \tilde S_{-a,0}(v) S_{a,0}(u) - S_{0,a}(u) \tilde S_{0, -a}(v)\right) \\
& +\tfrac{u-v}{(u+v+1)(u-v-1)} \left(
\tilde S_{0,a}(v) S_{0,-a}(u) 
-
S_{-a,0}(u) \tilde S_{a,0}(v)  
\right).
\end{align*}
Thus in order to show that
$[\tilde G(v), F_{0,a}(u) D_{a,b}(u) E_{b,0}(u)] -[\tilde S_{0,0}(v), S_{0,0}(u)]  =  0$, it suffices to show that
\[
\tfrac{1}{u+v+2} \left( \tilde S_{-a,0}(v) S_{a,0}(u) - S_{0,a}(u) \tilde S_{0, -a}(v)\right) 
 +\tfrac{1}{u+v+1} \left(
\tilde S_{0,a}(v) S_{0,-a}(u) 
-
S_{-a,0}(u) \tilde S_{a,0}(v)  
\right) = 0.
\]

Next, using \eqref{EQ:relhelp} with $i=0,j=a,k=0,l=-a$ (where we are summing over $a \in \{\pm 1\}$) gives
\begin{align} \label{EQ:relnmaa}
 (v^2&-u^2) \tilde S_{0,-a}(v) S_{0,a}(u) 
 - 2 (u-v) \left( \tilde S_{0,b}(v) S_{0,-b}(u) + \tilde S_{0,0}(v) S_{0,0}(u) \right) \\ \notag
 &\qquad +2 \left(\tilde S_{0,b}(v) S_{b,0}(u)   + \tilde S_{0,0}(v) S_{0,0}(u)\right)  \\ \notag
&= (v^2-u^2) S_{0,a}(u) \tilde S_{0,-a}(v) 
- (u-v) \left( S_{-b,a}(u) \tilde S_{b,-a}(u) + S_{0,a}(u) \tilde S_{0,-a}(v) \right) \\ \notag
&\qquad+ S_{-a,b}(u) \tilde S_{b,-a}(v) + S_{-a,0}(u) \tilde S_{0,-a}(v).
\end{align}
Since we are summing over $a,b \in \{\pm 1\}$, we can make the following substitutions:
\begin{itemize}
\item
$\tilde S_{0,b}(v) S_{0,-b}(u) = \tilde S_{0,-a}(v) S_{0,a}(u)$, 
\item
$\tilde S_{0,b}(v) S_{b,0}(u) = \tilde S_{0,a}(v) S_{a,0}(u)$, 
\item
$S_{-b,a}(u) \tilde S_{b,-a}(v) = S_{a,b}(u) \tilde S_{-a,-b}(v)$, 
\item
$S_{-a,b}(u) \tilde S_{b,-a}(v) = S_{a,b}(u) \tilde S_{b,a}(v)$, and  
\item
$S_{-a,0}(u) \tilde S_{0,-a}(v) = S_{a,0}(u) \tilde S_{0,a}(v)$.
\end{itemize}
Now plugging these into \eqref{EQ:relnmaa} and solving for $S_{0,a}(u) \tilde S_{0,-a}(v)$ gives
\begin{align*} 
  S_{0,a}(u) \tilde S_{0,-a}(v) &= \tfrac{u+v+2}{u+v+1} \tilde S_{0,-a}(v) S_{0,a}(u)
  +\tfrac{2(u-v-1)}{(u+v+1)(u-v)} \tilde S_{0,0}(v) S_{0,0}(u)  \\ \notag
  &\quad - \tfrac{2}{(u-v)(u+v+1)} \tilde S_{0,a}(v) S_{a,0}(u) 
   - \tfrac{1} {u+v+1} S_{a,b}(u) \tilde S_{-a,-b}(v) \\ \notag
  &\quad + \tfrac{1}{(u-v)(u+v+1)} S_{a,b}(u) \tilde S_{b,a}(v) 
  + \tfrac{1}{(u-v)(u+v+1)} S_{a,0}(u) \tilde S_{0,a}(v).
\end{align*}
Applying the anti-automorphism $\tau$ to this gives
\begin{align*}
  \tilde S_{-a,0}(v) S_{a,0}(u)  &= 
  \tfrac{u+v+2}{u+v+1} S_{a,0}(u) \tilde S_{-a,0}(v)
  +\tfrac{2(u-v-1)}{(u+v+1)(u-v)} S_{0,0}(u) \tilde S_{0,0}(v) \\
  &\quad - \tfrac{2}{(u-v)(u+v+1)} S_{0,a}(u) \tilde S_{a,0}(v) 
   - \tfrac{1} {u+v+1} \tilde S_{-a,-b}(v) S_{a,b}(u)  \\
  &\quad + \tfrac{1}{(u-v)(u+v+1)} \tilde S_{b,a}(v)  S_{a,b}(u)
  + \tfrac{1}{(u-v)(u+v+1)} \tilde S_{a,0}(v) S_{0,a}(u).
\end{align*}

So
\begin{align} \label{EQ:solveX}
\tfrac{1}{u+v+2} &\left( \tilde S_{-a,0}(v) S_{a,0}(u) - S_{0,a}(u) \tilde S_{0, -a}(v)\right) 
 +\tfrac{1}{u+v+1} \left(
\tilde S_{0,a}(v) S_{0,-a}(u) 
-
S_{-a,0}(u) \tilde S_{a,0}(v)  
\right)  \\ \notag
& = 
\tfrac{1}{u+v+2} \left(  
  \tfrac{u+v+2}{u+v+1} S_{a,0}(u) \tilde S_{-a,0}(v)
  +\tfrac{2(u-v-1)}{(u+v+1)(u-v)} S_{0,0}(u) \tilde S_{0,0}(v) \right.\\ \notag
  &\qquad \quad - \tfrac{2}{(u-v)(u+v+1)} S_{0,a}(u) \tilde S_{a,0}(v) 
   - \tfrac{1} {u+v+1} \tilde S_{-a,-b}(v) S_{a,b}(u)  \\ \notag
  & \qquad \quad + \tfrac{1}{(u-v)(u+v+1)} \tilde S_{b,a}(v)  S_{a,b}(u)
  + \tfrac{1}{(u-v)(u+v+1)} \tilde S_{a,0}(v) S_{0,a}(u) \\ \notag
& \qquad \quad - \left( 
  \tfrac{u+v+2}{u+v+1} \tilde S_{0,-a}(v) S_{0,a}(u)
  +\tfrac{2(u-v-1)}{(u+v+1)(u-v)} \tilde S_{0,0}(v) S_{0,0}(u)  \right. \\ \notag
  &\qquad \quad \quad - \tfrac{2}{(u-v)(u+v+1)} \tilde S_{0,a}(v) S_{a,0}(u) 
   - \tfrac{1} {u+v+1} S_{a,b}(u) \tilde S_{-a,-b}(v) \\ \notag
  &\qquad \quad \left . \left . \quad + \tfrac{1}{(u-v)(u+v+1)} S_{a,b}(u) \tilde S_{b,a}(v) 
  + \tfrac{1}{(u-v)(u+v+1)} S_{a,0}(u) \tilde S_{0,a}(v)\right) \right) \\ \notag
 &\qquad +\tfrac{1}{u+v+1} \left(
\tilde S_{0,a}(v) S_{0,-a}(u) 
-
S_{-a,0}(u) \tilde S_{a,0}(v)  \right) \\ \notag
&=
- \tfrac{2(u-v-1)}{(u+v+2)(u+v+1)(u-v)} [\tilde S_{0,0}(v), S_{0,0}(u)] \\ \notag
&\quad  + \tfrac{1}{(u+v+1)(u+v+2)} [ S_{a,b}(u) ,\tilde S_{-a,-b}(v)] \\ \notag
&\quad + \tfrac{1}{(u-v)(u+v+1)(u+v+2)} \left (-2 S_{0,a}(u) \tilde S_{a,0}(v) + \tilde S_{a,0}(v) S_{0,a}(u) \right. \\ \notag
&\qquad \qquad \qquad \qquad \qquad \quad \left. +2 \tilde S_{0,a}(v) S_{a,0}(u) -  S_{a,0}(u) \tilde S_{0,a}(v) +[S_{a,b}(u), \tilde S_{b,a}(v)] \right).
\end{align}

Next we need a formula for $[ S_{a,b}(u) ,\tilde S_{-a,-b}(v)]$.  Using \eqref{EQ:relhelp} with $i=a,j=b,k=-a,l=-b$, we have
\begin{align*}
   (v^2&-u^2)  \tilde S_{-a,-b}(v) S_{a,b}(u) - 2 (u-v) \left(\tilde S_{-a,c}(v) S_{a,-c}(u) + \tilde S_{-a,0}(v) S_{a,0}(u)\right) \\
   & \quad + (u+v) \left(\tilde S_{-a,c}(v) S_{c,-a} + \tilde S_{-a,0}(v) S_{0,-a}(u) \right)
   + 2 \left( \tilde S_{-a,c}(v) S_{c,-a}(u) + \tilde S_{-a,0}(v) S_{0,-a}(u) \right) \\
   & =
   (v^2-u^2) S_{a,b}(u) \tilde S_{-a,-b}(v) - 2 (u-v) \left( S_{-c,b}(u) \tilde S_{c,-b}(v) + S_{0,b}(u) \tilde S_{0,-b}(v) \right) \\
    &\quad + (u+v) \left(S_{a,c}(u) \tilde S_{c,a}(v) + S_{a,0}(u) \tilde S_{0,a}(v) \right)
    + 2 \left( S_{-b,c}(u) \tilde S_{c, -b}(v) + S_{-b,0}(u) S_{0,-b}(v) \right).
\end{align*}
Using the fact that we are summing over $a,b,c \in \{\pm 1\}$, this equation can be rewritten as 
\begin{align*}
   (v^2&-u^2)  \tilde S_{-a,-b}(v) S_{a,b}(u) - 2 (u-v) \left(\tilde S_{-a,-b}(v) S_{a,b}(u) + \tilde S_{-a,0}(v) S_{a,0}(u)\right) \\
   & \quad + (u+v) \left(\tilde S_{b,a}(v) S_{a,b}(u) + \tilde S_{a,0}(v) S_{0,a}(u) \right)
   + 2 \left( \tilde S_{b,a}(v) S_{a,b}(u) + \tilde S_{a,0}(v) S_{0,a}(u) \right) \\
   & =
   (v^2-u^2) S_{a,b}(u) \tilde S_{-a,-b}(v) - 2 (u-v) \left( S_{a,b}(u) \tilde S_{-a,-b}(v) + S_{0,a}(u) \tilde S_{0,-a}(v) \right) \\
    &\quad + (u+v) \left(S_{a,b}(u) \tilde S_{b,a}(v) + S_{a,0}(u) \tilde S_{0,a}(v) \right)
    + 2 \left( S_{a,b}(u) \tilde S_{b, a}(v) + S_{a,0}(u) \tilde S_{0,a}(v) \right),
\end{align*}
which we can solve for $[S_{a,b}(u), \tilde S_{-a,-b}(v)]$ to get
\begin{align*}
[S_{a,b}(u), \tilde S_{-a,-b}(v)] &= \tfrac{2}{u+v+2} \left( \tilde S_{-a,0}(v) S_{a,0}(u) - S_{0,a}(u) \tilde S_{0,-a}(v) \right)  \\
&\quad + \tfrac{1}{u-v} \left( S_{a,0}(u) \tilde S_{0,a}(v) - \tilde S_{a,0}(v) S_{0,a}(u) \right) \\
&\quad + \tfrac{1}{u-v} [S_{a,b}(u), \tilde S_{b,a}(v)].
\end{align*}

Now we plug this and \eqref{EQ:s0comm} into \eqref{EQ:solveX} and after canceling like terms we get
\begin{align*}
\tfrac{1}{u+v+2} &\left( \tilde S_{-a,0}(v) S_{a,0}(u) - S_{0,a}(u) \tilde S_{0, -a}(v)\right) 
 +\tfrac{1}{u+v+1} \left(
\tilde S_{0,a}(v) S_{0,-a}(u) 
-
S_{-a,0}(u) \tilde S_{a,0}(v)  
\right)  \\ \notag
&= 
\tfrac{2}{(u+v+2)^2(u+v+1)} \left( \tilde S_{-a,0}(v) S_{a,0}(u) - S_{0,a}(u) \tilde S_{0, -a}(v)\right)  \\
 &\qquad \qquad +\tfrac{2}{(u+v+2)(u+v+1)^2} \left(
\tilde S_{0,a}(v) S_{0,-a}(u) 
-
S_{-a,0}(u) \tilde S_{a,0}(v)  
\right).
\end{align*}
Now if we let 
\[X= 
\tfrac{1}{u+v+2} \left( \tilde S_{-a,0}(v) S_{a,0}(u) - S_{0,a}(u) \tilde S_{0, -a}(v)\right) 
 +\tfrac{1}{u+v+1} \left(
\tilde S_{0,a}(v) S_{0,-a}(u) 
-
S_{-a,0}(u) \tilde S_{a,0}(v)  
\right),
\]
then the previous equation is
\[
   X = \frac{2}{(u+v+2)(u+v+1)} X,
\]
which implies $X = 0$ as required.
\end{proof}

\subsection{Proof of Theorem \ref{T:1} \eqref{EQ:EFsymm} and \eqref{EQ:FEsymm}}
\begin{proof}
	By multiplying both sides of Theorem \ref{T:1} \eqref{EQ:GFactual} by $u+v+2$ we get
\begin{align*} 
     (u+v+2) [F_{i}(u), G(v)] &=  \tfrac{u+v+2}{u-v} (F_{i}(u) - F_{i}(v))G(v) 
     + G(v) (F_{i}(u) -E_{-i}(v)) .
     \end{align*}
     
     Now multiplying both sides on the left by $\tilde G(v)$, then
substituting $u-2$ for $v$ and $-u$ for $u$ gives
\[
 F_i(-u) = E_{-i} (u-2),
\]
which is Theorem \ref{T:1} \eqref{EQ:FEsymm}.
Now applying the anti-automorphism $\tau$ to this gives Theorem \ref{T:1} \eqref{EQ:EFsymm}.

\end{proof}

\section{Shifted twisted Yangians} \label{S:partialeval}
Fix a positive integer $k$.
The shifted twisted Yangian is the subalgebra of the twisted Yangian
generated by
\[
   \{ D_{i,j}^{(s)}, G^{(s)}, E_{l}^{(r)} \mid i,j,l \in \{\pm 1 \}, s \geq 1, r > k \}
\]
It is easy to see from the above relations that this is a subalgebra.
Let $Y_3^+(k)$ denote the shifted twisted Yangian for $\so_3(\C)$.

For all $k >1$ we define a homomorphism
\[
  \phi_k : Y_3^+(k) \to Y_3^+(k-1) \otimes U(\gl_1(\C))
\]
via
\[
\phi_k (D_{i,j}^{(m)}) = D_{i,j}^{(m)} \otimes 1,
\]
\[
\phi_k (E_{i}^{(m)}) = E_{i}^{(m)} \otimes 1 + E_{i}^{(m-1)} \otimes e_{0,0},
\]
and
\[
  \phi_k (G^{(m)}) = 
  G^{(m)} \otimes 1 
  - \left(\sum_{l=1}^{m-1} (-2)^l G^{(m-l-1)} \right) \otimes e_{0,0}
  - \left(\sum_{l=0}^{m-2} (-2)^l G^{(m-l-2)} \right) \otimes e_{0,0}^2
\]
We call the homomorphism $\phi_k$ {\em partial-evaluation homomorphisms}.

\begin{Theorem}
   The map $\phi_k$ is an algebra homomorphism.
\end{Theorem}
   \begin{proof}
   We need to check that $\phi_k$ respects the relations  \eqref{EQ:DEreln}, \eqref{EQ:GEreln}, and \eqref{EQ:EEreln}.
   It is straight forward to use \eqref{EQ:DEreln} to check that $\phi_k([D_{i,j}^{(m)}, E_{l,0}^{(n)}]) = [\phi_k(D_{i,j}^{(m)}), \phi_k(E_{l,0}^{(n)})]$.
   
   Next, using \eqref{EQ:GEreln} and letting $p=r-s$ and $q=s$, 
   we calculate that
   \begin{align*}
   \phi_k&([G^{(m)},E_i^{(n)}]) = \\
      &- \sum_{p=0}^{m-1} \left(G^{(p)} \otimes 1   
     - \bigg(\sum_{l=1}^{p-1} (-2)^l G^{(p-l-1)} \bigg) \otimes e_{0,0}
  - \bigg(\sum_{l=0}^{p-2} (-2)^l G^{(p-l-2)} \bigg) \otimes e_{0,0}^2 \right) \\
   &\qquad \qquad \times \left(E_i^{(m+n-1-p)} \otimes 1 + E_i^{(m+n-2-p)} \otimes e_{0,0} \right )  \\
   & + \sum_{p+q \leq m-1} (-2)^p(-1)^q \begin{pmatrix} p+q \\ q \end{pmatrix} \left(E_i^{(n+q)}\otimes 1 + E_i^{(n+q-1)} \otimes e_{0,0} \right)  \\
     &\qquad \times 
       \Bigg( 
      G^{(m-1-p-q)} \otimes 1 
  - \bigg(\sum_{l=1}^{m-2-p-q} (-2)^l G^{(m-l-2-p-q)} \bigg) \otimes e_{0,0} \\
  &\qquad \qquad \qquad - \bigg(\sum_{l=0}^{m-3-p-q} (-2)^l G^{(m-l-3-p-q)} \bigg) \otimes e_{0,0}^2
  \Bigg) 
  \end{align*}
  Note that the $\otimes e_{0,0}$ terms of this has coefficient $A+B+C+D$ where
  \[
  A = - \sum_{p=0}^{m-1} G^{(p)} E_i^{(m+n-2-p)} ,
  \]
  \[
  B = \sum_{p=0}^{m-1}\sum_{l=1}^{p-1} (-2)^l G^{(p-l-1)} E_i^{(m+n-1-p)},
  \]
  \[
   C = 
   \sum_{p+q \leq m-1}  
   (-2)^p(-1)^q \begin{pmatrix} p+q \\ q \end{pmatrix} 
   E_i^{(n+q-1)}  G^{(m-1-p-q)},
  \]
 and
  \[
  D = 
   -\sum_{p+q \leq m-1} 
   \sum_{l=1}^{m-2-p-q} 
   (-2)^p(-1)^q \begin{pmatrix} p+q \\ q \end{pmatrix} 
   (-2)^l E_i^{(n+q)} G^{(m-l-2-p-q)}.
   \]
   
   Note that 
   \begin{align*}
  B &= \sum_{p=2}^{m-1}\sum_{l=1}^{p-1} (-2)^l G^{(p-l-1)} E_i^{(m+n-1-p)}  \\
  &=
 \sum_{l=1}^{m-2}\sum_{p=l+1}^{m-1} (-2)^l G^{(p-l-1)} E_i^{(m+n-1-p)}  \\
 &= 
 \sum_{l=1}^{m-2}\sum_{p=0}^{m-l-2} (-2)^l G^{(p)} E_i^{(m+n-l-2-p)}.
 \end{align*}
 Furthermore,
\begin{align*}
  D &= 
   -\sum_{p+q \leq m-1} 
   \sum_{l=1}^{m-2-p-q} 
   (-2)^p(-1)^q \begin{pmatrix} p+q \\ q \end{pmatrix} 
   (-2)^l E_i^{(n+q)} G^{(m-l-2-p-q)} \\
   &=
   -\sum_{p+q \leq m-3} 
   \sum_{l=1}^{m-2-p-q} 
   (-2)^p(-1)^q \begin{pmatrix} p+q \\ q \end{pmatrix} 
   (-2)^l E_i^{(n+q)} G^{(m-l-2-p-q)} \\
   &=
   -\sum_{r=0}^{m-3}
   \sum_{p+q =r} 
   \sum_{l=1}^{m-2-r} 
   (-2)^p(-1)^q \begin{pmatrix} p+q \\ q \end{pmatrix} 
   (-2)^l E_i^{(n+q)} G^{(m-l-2-p-q)} \\
   &=
   -\sum_{r=0}^{m-3}
   \sum_{l=1}^{m-2-r} 
   \sum_{p+q =r} 
   (-2)^p(-1)^q \begin{pmatrix} p+q \\ q \end{pmatrix} 
   (-2)^l E_i^{(n+q)} G^{(m-l-2-p-q)} \\
   &=
   -\sum_{l=1}^{m-2}
   \sum_{r=0}^{m-2-l} 
   \sum_{p+q =r} 
   (-2)^p(-1)^q \begin{pmatrix} p+q \\ q \end{pmatrix} 
   (-2)^l E_i^{(n+q)} G^{(m-l-2-p-q)} \\
   &=
   -\sum_{l=1}^{m-2}
   \sum_{p+q \leq m-2-l } 
   (-2)^p(-1)^q \begin{pmatrix} p+q \\ q \end{pmatrix} 
   (-2)^l E_i^{(n+q)} G^{(m-l-2-p-q)}.
\end{align*}

  We also calculate that
  \begin{align*}
   [\phi_k(G^{(m)}), &\phi_k(E_i^{(n)})] =  \\
      & \Big[ 
  G^{(m)} \otimes 1 
  - \left(\sum_{l=1}^{m-1} (-2)^l G^{(m-l-1)} \right) \otimes e_{0,0} 
  - \left(\sum_{l=0}^{m-2} (-2)^l G^{(m-l-2)} \right) \otimes e_{0,0}^2,  \\
&\qquad      E_i^{(n)} \otimes 1 + E_i^{(n-1)} \otimes e_{0,0} \Big].
  \end{align*}
  Now we calculate that the coefficient of the $\otimes e_{0,0}$ term of this is
\begin{align*}
  -&\sum_{p=0}^{m-1} G^{(p)} E_i^{(m+n-2-p)} 
   + \sum_{p+q \leq m-1} (-2)^p (-1)^q 
   \begin{pmatrix} p+q \\ q \end{pmatrix}
   E_i^{(n+q-1)} G^{(m-1-p-q)} \\
  &\quad   + 
  \sum_{l=1}^{m-1} (-2)^l \sum_{p=0}^{m-l-2} G^{(p)} E_i^{(m+n-l-2-p)}  \\
   &\quad -
  \sum_{l=1}^{m-1} (-2)^l
  \sum_{p+q \leq m-l-2} (-2)^p (-1)^q 
  \begin{pmatrix} p+q \\ q \end{pmatrix} E_i^{(n+q)} G^{(m-l-2-p-q)}.
  \end{align*} 
 Now it is easy to see that this equals $A + C + B +D$.
 
 The calculation which shows that the $\otimes 1, \otimes e_{0,0}^2$, and $\otimes e_{0,0}^3$ terms from 
   $\phi_k([G^{(m)}, E_i^{(n)}])$ and
   $[\phi_k(G^{(m)}), \phi_k(E_i^{(n)})]$ agree is nearly identical, so we omit it.
   
   Next we need to show that
 \begin{equation} \label{EQ:EE}
    \phi_k([E_i^{(m)},E_j^{(n)}]) = [\phi_k(E_i^{(m)}),\phi_k(E_j^{(n)})].
 \end{equation}
 Using \eqref{EQ:EEreln} and the definition of $\phi_k$, we get that
 \begin{align} \label{EQ:LHS}
    \phi([E_i^{(m)},E_j^{(n)}]) &= 
    \sum_{r=0}^{n-1} \left( E_j^{(m+n-1-r)}\otimes 1 + E_j^{(m+n-2-r)} \otimes e_{0,0}\right) \left(E_i^{(r)}\otimes 1 + E_i^{(r-1)} \otimes e_{0,0}\right) \\ \notag
    & -\sum_{r=0}^{m-1} \left( E_j^{(m+n-1-r)}\otimes 1 + E_j^{(m+n-2-r)} \otimes e_{0,0}\right) \left(E_i^{(r)}\otimes 1 + E_i^{(r-1)} \otimes e_{0,0}\right) \\ \notag
    & - \sum_{r+s = m-1} (-2)^r (-1)^s \begin{pmatrix} m-1 \\ s \end{pmatrix} \\ \notag
    & \quad \times
      \sum_{p+q=n+s} \tilde D_{j,-i}^{(p)} 
      \left( G^{(q)} \otimes 1 
        - \sum_{l=1}^{q-1} (-2)^l G^{(q-l-1)} \otimes e_{0,0} \right. \\ \notag
   &\qquad \qquad \qquad \qquad \qquad \left.     - \sum_{l=0}^{q-2} (-2)^l G^{(q-l-2)} \otimes e_{0,0}^2 \right).
 \end{align}
 Also using the definition of $\phi_k$ we get
 \begin{equation} \label{EQ:RHS}
    [\phi_k(E_i^{(m)}),\phi_k(E_j^{(n)})] = 
    [E_i^{(m)} \otimes 1 + E_i^{(m-1)} \otimes e_{0,0}, 
      E_j^{(n)} \otimes 1 + E_j^{(n-1)} \otimes e_{0,0} ].
    \end{equation}
    Using \eqref{EQ:EEreln}, it is easy to see that the coefficients  of $\otimes 1$ in \eqref{EQ:LHS} and \eqref{EQ:RHS} agree.
    
    Now the coefficient of $\otimes e_{0,0}$ in \eqref{EQ:LHS} is
    \begin{align} \label{EQ:LHS2}
     &\sum_{r=0}^{n-1} \left(E_j^{(m+n-2-r)} E_i^{(r)} + E_j^{(m+n-1-r)} E_i^{(r-1)} \right) \\ \notag
     &\quad - \sum_{r=0}^{m-1} \left(E_j^{(m+n-2-r)} E_i^{(r)} + E_j^{(m+n-1-r)} E_i^{(r-1)} \right) \\ \notag
     &\quad + \sum_{r+s=m-1} (-2)^r (-1)^s \binom{m-1}{s} \sum_{p+q=n+s} D_{j,-i}^{(p)} \sum_{l=1}^{q-1} (-2)^l G^{(q-l-1)},
    \end{align}
    and using the definition of $\phi_k$ we calculate that the coefficient of $\otimes e_{0,0}$ in \eqref{EQ:RHS} is
    \begin{align} \label{EQ:RHS2}
        & \sum_{r=0}^{n-2} E_j^{(m+n-2-r)} E_i^{(r)} - \sum_{r=0}^{m-1} E_j^{(m+n-2-r)} E_i^{(r)} \\ \notag
        &\quad  - \sum_{r+s= m-1}(-2)^r(-1)^s \binom{m-1}{s} \sum_{p+q=n-1+s} \tilde D_{j,-i}^{(p)} G^{(q)} \\ \notag
        & \quad + \sum_{r=0}^{n-1} E_j^{(m+n-2-r)} E_i^{(r)} - \sum_{r=0}^{m-2} E_j^{(m+n-2-r)} E_i^{(r)} \\ \notag
        &\quad  - \sum_{r+s= m-2}(-2)^r(-1)^s \binom{m-2}{s} \sum_{p+q=n+s} \tilde D_{j,-i}^{(p)} G^{(q)}.
    \end{align}
Is is straight forward to check that sums of the terms involving $E_j^{(x)} E_i^{(y)}$ in \eqref{EQ:LHS2} and \eqref{EQ:RHS2} agree.
Now we consider the remaining terms in \eqref{EQ:LHS2}, ie we consider
    \begin{equation} \label{EQ:LHS3}
     \sum_{r+s=m-1} (-2)^r (-1)^s \binom{m-1}{s} \sum_{p+q=n+s} D_{j,-i}^{(p)} \sum_{l=1}^{q-1} (-2)^l G^{(q-l-1)},
    \end{equation}

We define the degree of $\tilde D_{-j,i}^{(x)} G^{(y)}$ to be $x+y$.  Now we observe that only terms with degrees between $0$ and $m+n-3$ inclusive occur.
Note that the in order to get a term of degree $x$, we must have that $l=n+s-x-1$.  Furthermore, since $l \geq 1$, we see that such an $l$ will only
occur when $s \geq x-n+2$.  Also since $q-1 \geq l =n+s-x-1$, such an $l$ will only occur  when $q \geq n+s-x$.
So the degree $x$ terms in \eqref{EQ:LHS3} are
\begin{align*}
   &\sum_{s=2+x-n}^{m-1} (-2)^{m-1-s} (-1)^s \binom{m-1}{s} \sum_{q=n+s-x}^{n+s} \tilde D_{j,-i}^{(n+s-q)} (-2)^{n+s-x-1}G^{(q+x-n-s)} = \\
   &\qquad 
   (-2)^{m+n-x-2} \sum_{s=0}^{m-1} (-1)^s \binom{m-1}{s} \sum_{p+q=x}\tilde D_{j,-i}^{(p)} G^{(q)}.
\end{align*}
Let $S_x = \sum_{y+z=x}\tilde D_{j,-i}^{(y)} G^{(z)}$, so we have that the degree $x$ terms in \eqref{EQ:LHS3} are equal to
\[
   (-2)^{m+n-x-2} \sum_{s=2+x-n}^{m-1} (-1)^s \binom{m-1}{s} S_x .
   \]
   Using the classical identity $\sum_{l=0}^n (-1)^l \binom{n}{l} = 0$, we see that this is only non-zero when $n-1 \leq x \leq n+m-3$.
   In this case a straight forward inductive proof shows that 
\[
   \sum_{s=2+x-n}^{m-1} (-1)^s \binom{m-1}{s} = (-1)^{x-n} \binom{m-2}{x-n+1}.
   \]
   Thus the sum of the degree $x$ terms in \eqref{EQ:LHS3} is 
   \begin{equation} \label{EQ:LHS4}
   (-2)^{m+n-x-2} (-1)^{x-n} \binom{m-2}{x-n+1} S_x.
   \end{equation}
   Now the  remaining terms from \eqref{EQ:RHS2} which do not involve the $E_j^{(y)} E_i^{(z)}$ terms are 
   \begin{align*}
        &- \sum_{r+s= m-1}(-2)^r(-1)^s \binom{m-1}{s} S_{n-1+s} 
        \quad  - \sum_{r+s= m-2}(-2)^r(-1)^s \binom{m-2}{s} S_{n+s}.
   \end{align*}
   A degree $x$ term only occurs here if $n-1 \leq x \leq n+m-3$, in which case the sum of the degree $x$ terms is
   \begin{align*}
        &- (-2)^{m+n-x-2} (-1)^{x-n+1} \binom{m-1}{x-n+1} S_{x}  
         - (-2)^{m+n-x-2}(-1)^{x-n} \binom{m-2}{x-n} S_{x}  \\
& \qquad      =   (-2)^{m+n-x-2}(-1)^{x-n} \left(\binom{m-1}{x-n+1} - \binom{m-2}{x-n} \right) S_x \\
&\qquad  =    (-2)^{m+n-x-2} (-1)^{x-n} \binom{m-2}{x-n+1} S_x.
   \end{align*}
   Comparing this with \eqref{EQ:LHS4}, we see that the $\otimes e_{0,0}$ terms on both sides of
   \eqref{EQ:EE}
    agree.
    
    The calculation which shows that the $\otimes e_{0,0}^2$ terms on both sides of \eqref{EQ:EE} agree is similar to the 
    $\otimes e_{0,0}$ case, so we omit it.
   \end{proof}

Fix positive integers $k,n$,
and let $N = 3n+2k$.
Let $\g = \sp_{N} (\C)$ if $N$ is even, and let $\g = \so_{N}(\C)$ if $N$ is odd.
To complete this section we will construct the homomorphism $\phi : Y_3^+ \to U(\g)$ from the introduction.

Now $\g$ is a subalgeba of $\gl_N(\C)$, which we consider to be generated by the matrix units $\{e_{i,j} \mid i,j \in \mathcal{I}_N\}$.
For any integer $i$ let 
\[
  \sigma_i = \begin{cases}
               1  & \text{if $N$ is even or $i \geq 0$}; \\
               -1 & \text{if $N$ is odd and $i < 0$}.
           \end{cases}
\]
Also  for $i,j \in \mathcal{I}_N$ let
\[
  f_{i,j} = e_{i,j} - \sigma_i \sigma_j e_{-j,-i}.
\]

Now $\g$ can be realized as the subalgebra of $\gl_N(\C)$ generated by $\{f_{i,j} \mid i,j \in \mathcal{I}_N\}$.

To define $\phi$ we use a three-row box diagram $\pi$ with $n$ rows in the top and bottom rows and $2k+n$ boxes in the middle row arranged symmetrically.
Furthermore we label the boxes of $\pi$ with the elements of $\mathcal{I}_N$ is a skew-symmetric manner.
For example if $n=2$, $k=2$, then one choice of labeling is
 \begin{equation*}
 \pi = 
    \begin{array}{c}
\begin{picture}(120,60)
\put(40,0){\line(1,0){40}} 
\put(0,20){\line(1,0){120}}
\put(0,40){\line(1,0){120}} 
\put(40,60){\line(1,0){40}}
\put(0,20){\line(0,1){20}} 
\put(20,20){\line(0,1){20}} 
\put(40,0){\line(0,1){60}} 
\put(60,0){\line(0,1){60}} 
\put(80,0){\line(0,1){60}} 
\put(100,20){\line(0,1){20}} 
\put(120,20){\line(0,1){20}} 
\put(8,30){\makebox(0,0){{-5}}} 
\put(28,30){\makebox(0,0){{-4}}} 
\put(48,50){\makebox(0,0){{-3}}} 
\put(48,30){\makebox(0,0){{-2}}} 
\put(48,10){\makebox(0,0){{-1}}} 
\put(70,50){\makebox(0,0){{1}}} 
\put(70,30){\makebox(0,0){{2}}} 
\put(70,10){\makebox(0,0){{3}}} 
\put(90,30){\makebox(0,0){{4}}} 
\put(110,30){\makebox(0,0){{5}}} 
\end{picture}
\end{array}
\end{equation*}

Now let $\mathfrak{l}$ be the Levi subalgebra of $\g$ spanned by $\{f_{i,j} \mid i,j$ are in the same column $\}$.
Note that 
\[
\mathfrak{l} \cong 
\begin{cases}
\gl_3(\C)^{\oplus n/2} \bigoplus \gl_1(\C)^{\otimes k} & \text{ if $n$ is even};  \\
\so_3(\C) \bigoplus \gl_3(\C)^{\oplus (n-1)/2} \bigoplus \gl_1(\C)^{\otimes k} & \text{ if $n$ is odd.}
\end{cases}
\]

Now $\phi_1 \circ \phi_2 \circ \dots \circ \phi_k: Y_3^+(k) \to Y_3^+ \otimes U(\gl_1(\C))^{\otimes k}$.
In this and is subsequent formula we are abusing notation somewhat by considering 
\[
\phi_i : Y_3^+ \otimes U(\gl_1(\C))^{\otimes j} \to
Y_3^+ \otimes U(\gl_1(\C))^{\otimes j+1} 
\]

to be
$\phi_i \otimes 1 ^{\otimes j}$.
In \cite[(1.12)]{Br} a map $\kappa_n$ is defined where
\[
  \kappa_n : Y_3^+ \to U(\gl_3(\C))^{\oplus n/2} 
  \]
  if $n$ is even and
  \[
  \kappa_n : Y_3^+ \to U(\so_3(\C)) \bigoplus U(\gl_3(\C))^{\oplus (n-1)/2}
  \]
  if $n$ is odd.
  
  Finally we define the map $\pi$ from the introduction via
  \[
\pi = \kappa_n \circ \phi_1 \circ \phi_2 \circ \dots \circ \phi_k: Y_3^+(k) \to U(\mathfrak{l}) \subseteq U(\g). 
\]

\section{The center of $Y_3^+$} \label{S:center}
An explicit formula for the center of all twisted Yangians is given in \cite{M3}.
We anticipate that the center will be crucially important in the study of the related finite $W$-algebras, so we express it here in terms of our new generators.
One potential advantage of the Drinfeld generators is that 
formula for the center of $Y^+_3$ is substantially simpler in terms of the Drinfeld generators.

In \cite{M3} the center of $Y^+_n$ is defined in terms of a the Sklyanin Determinant $\operatorname{sdet} S(u) \in Y^+_n[[u^{-1}]]$.
Now \cite[Theorem 3.14]{M3} says that the coefficients 
$\operatorname{sdet} S(u)$ 
are in the center of $Y^+_n$, and moreover the coefficients of the even powers of 
$\operatorname{sdet} S(u)$ 
generate the center of $Y^+_3$ and 
are algebraically independent.

Let 
\begin{equation}
C(u) = D_{-1,-1}(-u) D_{-1,-1}(u-1) G(u-2) 
- D_{-1,1}(-u) D_{1,-1}(u-1) G(u-2).
\end{equation}

\begin{Theorem}
The series $C(u)$ equals the Sklyanin Determinant $\operatorname{sdet} S(u)$ for $Y^+_3$.
Thus all of the coefficients of $C(u)$ lie in the center of $Y^+_3$, and moreover the coefficients of the even powers
of $C(u)$ generate the center of $Y^+_3$ and are algebraically independent.
\end{Theorem}
\begin{proof}
   By \cite[Theorem 3.12]{M3}, where we order indices $(-1,1,0)$, we have that 
   \begin{align} \label{EQ:tmp35}
      \operatorname{sdet} S(u) &= 
      S_{-1,-1}(-u) S_{-1,-1}(u-1) S_{0,0}(u-2)
      - S_{-1,-1}(-u) S_{0,-1}(u-1) S_{-1,0}(u-2) \\ \notag
      &\quad - S_{-1,1}(-u) S_{1,-1}(u-1) S_{0,0}(u-2)
      + S_{1,0}(-u) S_{-1,1}(u-1) S_{1,0}(u-2) \\ \notag
      &\quad - S_{1,1}(-u) S_{0,1}(u-1) S_{1,0}(u-2) 
     + S_{-1,0}(-u) S_{1,-1}(u-1) S_{-1,0}(u-2).
   \end{align}
   We examine the six terms here one at a time.  
   
   First we examine the first term.  
   By \eqref{EQ:Smatrix2}
  \begin{align} \label{EQ:tmp30}
      S_{-1,-1}(-u) &S_{-1,-1}(u-1) S_{0,0}(u-2) = \\ \notag 
      &D_{-1,-1}(-u) D_{-1,-1}(u-1) (G(u-2) + F_{a}(u-2) D_{a,b}(u-2)E_b(u-2)),
  \end{align}
   where we are summing over $a,b \in \{\pm 1\}$ as usual.
   By Theorem \ref{T:1} \eqref{EQ:EFsymm}, $F_a(u-2) = E_{-a}(-u)$.
   Furthermore, by substituting $u$ with $u-1$ and $v$ with $-u$ in Theorem \ref{T:1} \eqref{EQ:DE} we get that
  \begin{align}
   D_{-1,-1}(u-1) E_{-a}(-u) &= E_{-a}(-u) D_{-1,-1}(u-1) + \delta_{-1,a}(F_c(u-1)-E_{-c}(-u)) D_{c,-1}(u-1) \\
   &\quad + \frac{\delta_{1,a}}{2u-1} D_{-1,c}(u-1) (E_c(-u) - E_c(u-1)),
  \end{align}
  where we are summing over $c \in \{\pm 1\}$.
  Plugging this into \eqref{EQ:tmp30} gives
  \begin{align} \label{EQ:tmp31}
      S_{-1,-1}(-u) &S_{-1,-1}(u-1) S_{0,0}(u-2) =  \\ 
      &D_{-1,-1}(-u) D_{-1,-1}(u-1) G(u-2) \\
   & +D_{-1,-1}(-u) E_{-a}(-u) D_{-1,-1}(u-1) D_{a,b}(u-2) E_b(u-2) \\
 &+ D_{-1,-1}(-u) F_{c}(u-1) D_{c,-1}(u-1) D_{-1,b}(u-2) E_b(u-2) \\
  &- D_{-1,-1}(-u) E_{-c}(-u) D_{c,-1}(u-1) D_{-1,b}(u-2) E_b(u-2) \\
   &+ \frac{1}{2u-1} D_{-1,-1}(-u) D_{-1,c}(u-1) (E_c(-u) - E_c(u-1))D_{1,b}(u-2) E_b(u-2).
  \end{align}
  
  Next we look at the second term of \eqref{EQ:tmp35}.
   By \eqref{EQ:Smatrix2}
  \begin{align} \label{EQ:tmp32}
      &- S_{-1,-1}(-u) S_{0,-1}(u-1) S_{-1,0}(u-2)  =  \\ \notag
      &\qquad \qquad \qquad -D_{-1,-1}(-u) F_a(u-1) D_{a,-1}(u-1) D_{-1,b}(u-2) E_b(u-2).
  \end{align}

  Next we look at the third term of \eqref{EQ:tmp35}.
   By \eqref{EQ:Smatrix2}
   \begin{align} \label{EQ:tmp36}
      &- S_{-1,1}(-u) S_{1,-1}(u-1) S_{0,0}(u-2) = \\ \notag
      &\qquad \qquad \qquad -D_{-1,1}(-u) D_{1,-1}(u-1)\left(G(u-2) + F_{a}(u-2) D_{a,b}(u-2) E_b(u-2)\right).
   \end{align}
   Now by using  Theorem \ref{T:1} \eqref{EQ:EFsymm} and Theorem \ref{T:1} \eqref{EQ:DE} we get that
   \begin{align} \notag
     &D_{1,-1}(u-1) F_a(u-2) = D_{1,-1}(u-1) E_{-a}(-u) =  \\ \notag
     &\qquad E_{-a}(-u)D_{1,-1}(u-1) + \delta_{1,a} \left(F_c(u-1)-E_{-c}(-u)\right) D_{c,-1}(u-1) \\ \notag
     &\qquad + \frac{\delta_{1,a}}{2u-1} D_{1,c}(u-1) (E_c(-u)-E_c(u-1)).
   \end{align}
   Plugging this into \eqref{EQ:tmp36} gives
   \begin{align} \label{EQ:tmp37}
      - S_{-1,1}(-u) &S_{1,-1}(u-1) S_{0,0}(u-2) = \\ 
   & -D_{-1,1}(-u) D_{1,-1}(u-1)G(u-2)  \\
   & -D_{-1,1}(-u) E_{-a}(-u) D_{1,-1}(u-1)D_{a,b}(u-2)E_b(u-2)  \\
   & -D_{-1,1}(-u) F_{c}(u-1) D_{c,-1}(u-1)D_{1,b}(u-2)E_b(u-2)  \\
   & +D_{-1,1}(-u) E_{-c}(-u) D_{c,-1}(u-1)D_{1,b}(u-2)E_b(u-2)  \\
  & - \frac{1}{2u-1} D_{-1,1}(-u) D_{1,c}(u-1) (E_c(-u)-E_c(u-1))D_{1,b}(u-2) E_b(u-2).
      \end{align}

  Next we look at the fourth term of \eqref{EQ:tmp35}.
  First, using \eqref{EQ:SS}, we see that
  \begin{align} \label{EQ:tmp43}
     S_{1,0}(-u) &S_{-1,1}(u-1) = S_{-1,1}(u-1) S_{1,0}(-u)  \\ \notag
     &\quad + S_{1,1}(-u) S_{0,1}(u-1) - S_{-1,-1}(u-1) S_{-1,0}(-u) \\ \notag
     &\quad - \frac{1}{2u-1} (S_{-1,0}(-u) S_{1,1}(u-1) - S_{-1,0}(u-1) S_{1,1}(-u)) \\ \notag
     &\quad + \frac{1}{2u-1} (S_{-1,-1}(-u) S_{0,1}(u-1) - S_{-1,-1}(u-1) S_{0,1}(-u)).
  \end{align}
  Also using \eqref{EQ:SS} we have that
  \begin{align}
      S_{-1,-1}(u-1) &S_{-1,0}(-u)  = S_{-1,0}(-u) S_{-1,-1}(u-1) \\ \notag
      &\quad + S_{-1,1}(u-1) S_{1,0}(-u) - S_{-1,1}(-u) S_{0,-1}(u-1) \\ \notag
     &\quad + \frac{1}{2u-1} (S_{-1,-1}(u-1) S_{-1,0}(-u) - S_{-1,-1}(-u) S_{-1,0}(u-1)) \\ \notag
     &\quad - \frac{1}{2u-1} (S_{-1,1}(u-1) S_{1,0}(-u) - S_{-1,1}(-u) S_{1,0}(u-1))
     \end{align}
     Plugging this into \eqref{EQ:tmp43}, canceling terms, then multiplying on the right by $S_{1,0}(u-2)$ to get the fourth term of \eqref{EQ:tmp35} now gives
     \begin{align}  \notag
  S_{1,0}(-u) S_{-1,1}(u-1) &S_{1,0}(u-2)  = S_{1,1}(-u) S_{0,1}(u-1) S_{1,0}(u-2) \\ \notag
    &\quad - S_{-1,0}(-u) S_{-1,-1}(u-1) S_{1,0}(u-2) \\ \notag
    &\quad + S_{-1,1}(-u) S_{0,-1}(u-1) S_{1,0}(u-2) \\ \notag
    &\quad + \frac{1}{2u-1} \Big( 
     -S_{-1,0}(-u) S_{1,1}(u-1) + S_{-1,0}(u-1) S_{1,1}(-u) \\ \notag
     &\quad \qquad + S_{-1,-1}(-u) S_{0,1}(u-1) - S_{-1,-1}(u-1) S_{0,1}(-u) \\ \notag
     &\quad \qquad - S_{-1,-1}(u-1) S_{-1,0}(-u) + S_{-1,-1}(-u) S_{-1,0}(u-1) \\ \notag
     &\quad \qquad + S_{-1,1}(u-1) S_{1,0}(-u) - S_{-1,1}(-u) S_{1,0}(u-1) \Big ) S_{1,0}(u-2).
     \end{align}
     Now using \eqref{EQ:Smatrix2} we get that
     \begin{align}  \label{EQ:tmp44}
  S_{1,0}(-u) S_{-1,1}(u-1) &S_{1,0}(u-2)  = S_{1,1}(-u) S_{0,1}(u-1) S_{1,0}(u-2) \\ \notag
    &\quad - D_{-1,a}(-u) E_a(-u) D_{-1,-1}(u-1)  D_{1,b}(u-2) E_b(u-2)  \\ \notag
    &\quad + D_{-1,1}(-u) F_a(u-1) D_{a,-1}(u-1) D_{1,b}(u-2) E_b(u-2) \\ \notag
    &\quad + \frac{1}{2u-1} \Big( 
     -S_{-1,0}(-u) S_{1,1}(u-1) + S_{-1,0}(u-1) S_{1,1}(-u) \\ \notag
     &\quad \qquad + S_{-1,-1}(-u) S_{0,1}(u-1) - S_{-1,-1}(u-1) S_{0,1}(-u) \\ \notag
     &\quad \qquad - S_{-1,-1}(u-1) S_{-1,0}(-u) + S_{-1,-1}(-u) S_{-1,0}(u-1) \\ \notag
     &\quad \qquad + S_{-1,1}(u-1) S_{1,0}(-u) - S_{-1,1}(-u) S_{1,0}(u-1) \Big ) S_{1,0}(u-2).
     \end{align}
     
     We do not need to manipulate the fifth terms of \eqref{EQ:tmp35}, but for convenience we recall that this term is
     \begin{equation} \label{EQ:tmp45}
   -S_{1,1}(-u) S_{0,1}(u-1) S_{1,0}(u-2).
     \end{equation}
     
  Finally we look at the sixth term of \eqref{EQ:tmp35}.
   By \eqref{EQ:Smatrix2} we have that
  \begin{align} \label{EQ:tmp46}
     & S_{-1,0}(-u) S_{1,-1}(u-1) S_{-1,0}(u-2) = \\ \notag
      &\qquad \qquad \qquad D_{-1,a}(-u) E_a(-u) D_{1,-1}(u-1) D_{-1,b}(u-2) E_b(u-2).
  \end{align}
  
  So by examining \eqref{EQ:tmp30}, \eqref{EQ:tmp32}, \eqref{EQ:tmp37}, \eqref{EQ:tmp44}, \eqref{EQ:tmp45}, and \eqref{EQ:tmp46},
  and canceling as much as possible (using that in all of these we are summing over $a,b,c \in \{\pm1\}$)
  we see that \eqref{EQ:tmp35} is equal to
  \begin{align}
      D_{-1,-1}&(-u) D_{-1,-1}(u-1) G(u-2)  -D_{-1,1}(-u) D_{1,-1}(u-1)G(u-2)   \\ \notag
    & + \frac{1}{2u-1} A
     S_{1,0}(u-2).
  \end{align}
  where
  \begin{align*}
A&=   
   D_{-1,-1}(-u) D_{-1,c}(u-1) (E_c(-u) - E_c(u-1))  \\
   &\qquad -  D_{-1,1}(-u) D_{1,c}(u-1) (E_c(-u)-E_c(u-1)) \\ \notag
  &  \qquad -S_{-1,0}(-u) S_{1,1}(u-1) + S_{-1,0}(u-1) S_{1,1}(-u) \\
  &\qquad + S_{-1,-1}(-u) S_{0,1}(u-1)  - S_{-1,-1}(u-1) S_{0,1}(-u)\\
  &\qquad - S_{-1,-1}(u-1) S_{-1,0}(-u) + S_{-1,-1}(-u) S_{-1,0}(u-1) \\ \notag
     & \qquad + S_{-1,1}(u-1) S_{1,0}(-u) - S_{-1,1}(-u) S_{1,0}(u-1).
  \end{align*}
  So to prove that $C(u) = \operatorname{sdet}(u)$, we need to show that $A=0$.
  For convenience, we label these 6 lines $A_1, \dots A_6$.  
  Furthermore, after distributing, each line has 2 terms.  We refer to these as $A_{1,1}, A_{1,2}, A_{2,1}$, etc.
  So for example $A_{1,2} =  -D_{-1,1}(-u) D_{1,c}(u-1) E_c(u-1).$

  First note that 
  $A_{1,2}$ cancels with $A_{5,2}$ and $A_{2,2}$ cancels with $A_{6,2}$.
  
  Next note that both $A_{1,1}$ and $A_{2,1}$ contain terms of the form
   $ D_{x,c}(u-1) E_c{-u}$ where $x \in \{\pm1\}$ and we are summing over $c \in \{\pm 1\}$.
   Using Theorem \ref{T:1} \eqref{EQ:DE} we get that
  \begin{align*}
    D_{x,c}(u-1) E_c(-u) &= E_c(-u) D_{x,c}(u-1) + \tfrac{2}{2u-1} D_{x,b}(u-1)(E_b(-u)-E_b(u-1)) \\
       &\quad + (F_b(u-1) - E_{-b}(-u)) D_{b,-x}(u-1).
  \end{align*}
  Since we are summing over both $b$ and $c$ in $\{\pm 1\}$, we can replace the $b$'s with $c$'s and solve for $D_{x,c}(u-1)E_c(-u)$.
  Then we can replace $-c$ in the last term.  This gives
  This gives
  \begin{align*}
     D_{x,c}(u-1) E_c(-u) &= \tfrac{2u-1}{2u-3} \Big( E_c(-u) D_{x,c}(u-1) + (F_{-c}(u-1)-E_{c}(-u))D_{-c,-x}(u-1)\Big) \\
     &\quad - \tfrac{2}{2u-3} D_{x,c}(u-1) E_c(u-1).
  \end{align*}
  Next note that the two terms to the right of the equals sign with $E_c(-u)$ cancel when $c=-x$.
  Thus
  \begin{align*}
     D_{x,c}(u-1) E_c(-u) &= \tfrac{2u-1}{2u-3} \Big( E_x(-u) D_{x,x}(u-1) \\
     &\quad \qquad \qquad - E_x(-u) D_{-x,-x}(u-1) + F_{-c}(u-1)D_{-c,-x}(u-1)\Big) \\
     &\quad - \tfrac{2}{2u-3} D_{x,c}(u-1) E_c(u-1).
  \end{align*}
  So 
  \begin{align} \label{EQ:t50}
  A_{1,1}&+A_{2,1}  
     = \\ \notag
     &\tfrac{2u-1}{2u-3} \Big( 
     D_{-1,-1}(-u) E_{-1}(-u) D_{-1,-1}(u-1) 
     - D_{-1,-1}(-u) E_{-1}(-u) D_{1,1}(u-1)  \\ \notag
&\qquad \quad    -D_{-1,1}(-u) E_{1}(-u) D_{1,1}(u-1) 
     + D_{-1,1}(-u) E_{1}(-u) D_{-1,-1}(u-1)  \\ \notag
  &\qquad \quad  + 
  D_{-1,-1}(-u) F_{-c}(u-1)D_{-c,1}(u-1) -
  D_{-1,1}(-u) F_{-c}(u-1)D_{-c,-1}(u-1)\Big) \\ \notag
     &+ \tfrac{2}{2u-3} \Big(
     -D_{-1,-1}(-u) D_{-1,c}(u-1) E_c(u-1)
     +D_{-1,1}(-u) D_{1,c}(u-1) E_c(u-1) \Big) = \\ \notag
     &\tfrac{2u-1}{2u-3} \Big( 
     D_{-1,c}(-u) E_{c}(-u) D_{-1,-1}(u-1) 
     - D_{-1,c}(-u) E_{c}(-u) D_{1,1}(u-1)  \\ \notag
  &\qquad \quad  + 
  D_{-1,-1}(-u) F_{-c}(u-1)D_{-c,1}(u-1) -
  D_{-1,1}(-u) F_{-c}(u-1)D_{-c,-1}(u-1)\Big) \\ \notag
     &+ \tfrac{2}{2u-3} \Big(
     -D_{-1,-1}(-u) D_{-1,c}(u-1) E_c(u-1)
     +D_{-1,1}(-u) D_{1,c}(u-1) E_c(u-1) \Big) = \\ \notag
     &\tfrac{2u-1}{2u-3} \Big( 
     S_{-1,0}(-u) S_{-1,-1}(u-1) -
     S_{-1,0}(-u) S_{1,1}(u-1) \\ \notag
  &\qquad \quad  + 
  S_{-1,-1}(-u) S_{0,1}(u-1) -
  S_{-1,1}(-u) S_{0,-1}(u-1) \Big) \\ \notag
     &+ \tfrac{2}{2u-3} \Big(
     -S_{-1,-1}(-u) S_{-1,0}(u-1) +
     S_{-1,1}(-u) S_{1,0}(u-1)
     \Big),
  \end{align}
  where for the last equality we are using \eqref{EQ:Smatrix}
  
  Next we use \eqref{EQ:SS}
  to get
  \begin{align*}
   A_{3,2} &= S_{-1,0}(u-1) S_{1,1}(-u)  \\
   &= S_{1,1}(-u) S_{-1,0}(u-1) + S_{-1,-1}(u-1)S_{0,1}(-u) - S_{1,1}(-u)S_{-1,0}(u-1) + L_{3,2} \\
   &= S_{-1,-1}(u-1)S_{0,1}(-u) + L_{3,2},
  \end{align*}
  where
  \begin{align*}
  L_{3,2} = \tfrac{1}{2u-1} &\Big(
                  S_{-1,1}(-u) S_{1,0}(u-1) - S_{-1,1}(u-1) S_{1,0}(-u) \\
                  &\quad - S_{-1,-1}(-u) S_{-1,0}(u-1) + S_{-1,-1}(u-1) S_{-1,0}(-u) \Big).
  \end{align*}
  Here we are calculating $L_{3,2}$ by calculating the ``lower degree terms''
  of \eqref{EQ:SS} applied to $-[S_{1,1}(-u), S_{-1,0}(u-1)]$.  Note that if you change the order of these terms then 
  apply \eqref{EQ:SS} you will get something which appears to be rather different.
  
 So 
 \begin{align} \label{EQ:t47}
  A_{3,2} + A_{4,2} = 
    L_{3,2}.
 \end{align}
 
 Next we calculate using \eqref{EQ:SS} that
 \begin{align} \label{EQ:t49}
    A_{4,1} &= S_{-1,-1}(-u) S_{0,1}(u-1)  \\ \notag
    &= S_{0,1}(u-1) S_{-1,-1}(-u) +S_{-1,0}(-u) S_{1,1}(u-1) - S_{0,1}(u-1) S_{-1,-1}(-u) + L_{4,1} \\ \notag
    &= S_{-1,0}(-u) S_{1,1}(u-1) + L_{4,1},
 \end{align}
 where 
 \begin{align} \label{EQ:t48}
 L_{4,1}= \tfrac{1}{2u-1} &\Big( -S_{-1,1}(u-1) S_{0,-1}(-u) + S_{-1,1}(-u)S_{0,-1}(u-1) \\ \notag
  &\quad +S_{-1,0}(u-1) S_{-1,-1}(-u) - S_{-1,0}(-u) S_{-1,-1}(u-1)\Big).
 \end{align}
 We will simplify this by finding a nice expression for 
 \[
 S_{-1,0}(u-1)S_{-1,-1}(-u) - S_{-1,1}(u-1) S_{0,-1}(-u).
 \]
 By using \eqref{EQ:SS} on the first term, we get that
 \begin{align*}
 S_{-1,0}&(u-1)S_{-1,-1}(-u) - S_{-1,1}(u-1) S_{0,-1}(-u) = \\
 &S_{-1,-1}(-u) S_{-1,0}(u-1) - S_{-1,1}(-u) S_{1,0}(u-1) \\
 &\quad +\tfrac{1}{2u-1}\Big(S_{-1,0}(u-1) S_{-1,-1}(-u) - S_{-1,0}(-u) S_{-1,-1}(u-1)\\
 &\qquad \qquad \qquad -S_{-1,1}(u-1) S_{0,-1}(-u) + S_{-1,1}(-u) S_{0,-1}(u-1) \Big).
 \end{align*}
 By solving this equation for $ S_{-1,0}(u-1)S_{-1,-1}(-u) - S_{-1,1}(u-1) S_{0,-1}(-u)$ we get that 
 \begin{align*}
 S_{-1,0}(u-1)S_{-1,-1}(-u) &- S_{-1,1}(u-1) S_{0,-1}(-u) =  \\
 &\tfrac{2u-1}{2u-2} \Big( S_{-1,-1}(-u) S_{-1,0}(u-1) - S_{-1,1}(-u) S_{1,0}(u-1)\Big) \\
 &\quad + \tfrac{1}{2u-2}\Big(S_{-1,1}(-u) S_{0,-1}(u-1) - S_{-1,0}(-u) S_{-1,-1}(u-1)\Big).
 \end{align*}
 Plugging this into \eqref{EQ:t48} gives
 \begin{align*} 
 L_{4,1}&= \tfrac{1}{2u-1} \Big(  S_{-1,1}(-u)S_{0,-1}(u-1) - S_{-1,0}(-u) S_{-1,-1}(u-1)\Big) \\
  &\quad+\tfrac{1}{2u-2}
   \Big( S_{-1,-1}(-u) S_{-1,0}(u-1) - S_{-1,1}(-u) S_{1,0}(u-1)\Big) \\
 &\quad + \tfrac{1}{(2u-2)(2u-1)}\Big(S_{-1,1}(-u) S_{0,-1}(u-1) - S_{-1,0}(-u) S_{-1,-1}(u-1)\Big).
 \end{align*}
 Now using this and \eqref{EQ:t49} we have that
 \begin{align} \label{EQ:t53}
   A_{3,1} +A_{4,1} &=  S_{-1,-1}(-u)S_{0,1}(u-1) - S_{-1,0}(-u) S_{1,1}(u-1)  \\ \notag
   &=
 \tfrac{1}{2u-1} \Big(  S_{-1,1}(-u)S_{0,-1}(u-1) - S_{-1,0}(-u) S_{-1,-1}(u-1)\Big) \\ \notag
  &\quad+\tfrac{1}{2u-2}
   \Big( S_{-1,-1}(-u) S_{-1,0}(u-1) - S_{-1,1}(-u) S_{1,0}(u-1)\Big) \\ \notag
 &\quad + \tfrac{1}{(2u-2)(2u-1)}\Big(S_{-1,1}(-u) S_{0,-1}(u-1) - S_{-1,0}(-u) S_{-1,-1}(u-1)\Big).
 \end{align}
 Furthermore we can plug this into the expression in the last equality of \eqref{EQ:t50} to get that
 \begin{align} \label{EQ:t51}
    A_{1,1} + A_{2,1} = 
     &\tfrac{2u-1}{2u-3} \Big( 
     S_{-1,0}(-u) S_{-1,-1}(u-1) - S_{-1,1}(-u) S_{0,-1}(u-1)  \\ \notag
      &\qquad + 
 \tfrac{1}{2u-1} \Big(  S_{-1,1}(-u)S_{0,-1}(u-1) - S_{-1,0}(-u) S_{-1,-1}(u-1)\Big) \\ \notag
  &\qquad+\tfrac{1}{2u-2}
   \Big( S_{-1,-1}(-u) S_{-1,0}(u-1) - S_{-1,1}(-u) S_{1,0}(u-1)\Big) \\ \notag
 &\qquad + \tfrac{1}{(2u-2)(2u-1)}\Big(S_{-1,1}(-u) S_{0,-1}(u-1) - S_{-1,0}(-u) S_{-1,-1}(u-1)\Big)
     \Big) \\ \notag
     &+ \tfrac{2}{2u-3} \Big(
     -S_{-1,-1}(-u) S_{-1,0}(u-1) +
     S_{-1,1}(-u) S_{1,0}(u-1)
     \Big),
 \end{align}
 
 Finally we use \eqref{EQ:SS} to write
 \begin{align*}
  A_{5,1} &= -S_{-1,-1}(u-1) S_{-1,0}(-u) \\
   &= -S_{-1,0}(-u) S_{-1,-1}(u-1) - S_{-1,1}(u-1) S_{1,0}(-u) + S_{-1,1}(-u) S_{0,-1}(u-1)  + L_{5,1},
 \end{align*}
 where 
 \begin{align*}
  L_{5,1} &= \tfrac{1}{2u-1} \Big(-S_{-1,-1}(u-1) S_{-1,0}(-u) + S_{-1,-1}(-u) S_{-1,0}(u-1) \\
  &\qquad \qquad  + S_{-1,1}(u-1) S_{1,0}(-u) - S_{-1,1}(-u) S_{1,0}(u-1) \Big).
 \end{align*}
 So
 \begin{align} \label{EQ:t52}
   A_{5,1} + A_{6,1} = 
    -S_{-1,0}(-u) S_{-1,-1}(u-1) + S_{-1,1}(-u) S_{0,-1}(u-1)  + L_{5,1}.
 \end{align}
 
 Now to find $A = A_{1,1} + A_{2,1} + A_{3,1} + A_{3,2} + A_{4,1} + A_{4,2} + A_{5,1} + A_{6,1}$, first note
 that $L_{3,2} = - L_{5,1}$,
 then calculate the coefficients of 
 $S_{-1,0}(-u) S_{-1,-1}(u-1)$,
 $S_{-1,1}(-u) S_{0,-1}(u-1)$,
 $S_{-1,-1}(-u) S_{-1,0}(u-1)$,
 and
 $S_{-1,1}(-u) S_{1,0}(u-1)$
 in \eqref{EQ:t51}, \eqref{EQ:t53}, \eqref{EQ:t47}, and \eqref{EQ:t52}.
 This shows that $A=0.$

\end{proof}

From the theorem it immediately follows that
the center of of $Y_3^+$ is contained in every shifted twisted Yangian:
\begin{Corollary}
The center of $Y_3^+$ is contained in $Y_3^+(k)$ for all $k > 0$.
\end{Corollary}

Finally we conjecture that the center of every shifted twisted Yangian
is equal to the center of $Y_3^+$.

\end{document}